\documentclass{elsarticle}

\usepackage{amsmath}%
\usepackage{amsfonts}%
\usepackage{amssymb}%
\usepackage{float}
\usepackage{amsthm}
\usepackage{graphicx}
\usepackage{color}
\usepackage[normalem]{ulem}
\usepackage{pdfsync}

\numberwithin{equation}{section}

\newtheorem{theorem}{Theorem}
\newtheorem{corollary}{Corollary}
\newtheorem{lemma}{Lemma}
\newtheorem{proposition}{Proposition}

\newtheorem{remark}{Remark}
\newtheorem{algorithm}{Algorithm}

\newcommand{\ve}{\varepsilon}

\renewcommand{\epsilon}{\varepsilon}

\definecolor{lzcol}{rgb}{0,0,1}

\makeatletter
\renewcommand{\maketag@@@}[1]{\hbox{\m@th\normalsize\normalfont#1}}%
\makeatother

\begin{document}
\begin{frontmatter}

\title{On an evolution equation in a cell motility model}

\author[psu]{Matthew S. Mizuhara\corref{cor1}\fnref{fn1}}
\ead{msm344@psu.edu}

\author[psu]{Leonid Berlyand\fnref{fn2}}
\ead{lvb2@psu.edu}

\author[adr]{Volodymyr Rybalko\fnref{fn2}}
\ead{vrybalko@ilt.kharkov.ua}

\author[adz]{Lei Zhang\fnref{fn3}}
\ead{lzhang2012@sjtu.edu.cn}

\cortext[cor1]{Corresponding author}

\fntext[fn1]{This author was supported by the Department of Defense (DoD) through the National Defense Science \& Engineering Graduate Fellowship (NDSEG) Program. He also received partial support from NSF grants DMS-1106666 and DMS-1405769.}

\fntext[fn2]{This author was supported by NSF grants DMS-1106666 and DMS-1405769.}

\fntext[fn3]{The work of this author was supported by NSFC grant 11471214 and 1000 Plan for Young Scientists of China.}

\address[psu]{Department of Mathematics, The Pennsylvania State University, University Park, PA 16802, USA}

\address[adr]{Mathematical Division, B.Verkin Institute for Low Temperature Physics and Engineering of the National Academy of Sciences of Ukraine, 47 Lenin Ave., 61103 Kharkiv, Ukraine}

\address[adz]{Institute of Natural Sciences and Department of Mathematics, Shanghai Jiaotong University; Key Laboratory of Scientific and Engineering Computing (Shanghai Jiao Tong University), Ministry of Education, 800 Dongchuan Road, 200240, Shanghai}

\begin{abstract}
	
	This paper  deals with the evolution equation of a curve obtained as the sharp interface limit of a non-linear  system of two reaction-diffusion PDEs. This system was introduced as a phase-field model of (crawling) motion of eukaryotic cells on a substrate. The key issue is the evolution of the cell membrane (interface curve) which involves shape change and net motion. This issue can be addressed both qualitatively and quantitatively by studying the evolution equation of the sharp interface limit for this system. However, this equation is non-linear and non-local and  existence  of  solutions presents a significant analytical challenge. We establish existence of solutions for a wide class of initial data in  the so-called subcritical regime. Existence is proved in a two step procedure. First, for smooth ($H^2$) initial data  we use a regularization technique. Second, we consider non-smooth initial data that are more relevant from the application point of view. Here, uniform estimates on the  time when solution  exists rely on  a maximum principle type argument.   
We also explore the long time behavior of the model using both analytical and numerical tools. We prove the nonexistence of traveling wave solutions with nonzero velocity. Numerical experiments show that presence of non-linearity and asymmetry of the initial curve results in a net motion which distinguishes it from classical volume preserving curvature motion. This is done by developing an algorithm for efficient numerical resolution of the non-local term in the evolution equation.
\end{abstract}


\end{frontmatter}

\section{Introduction}

%

\subsection{Phase field PDE model of cell motility} 
This work is motivated by a 2D phase field model of crawling cell motility introduced in \cite{Zie12}. This model consists of a system of two PDEs for the phase field function and the orientation vector due to polymerization of actin filaments inside the cell. In addition it obeys a volume preservation constraint. In \cite{Ber14} this system was rewritten in the following simplified form suitable for asymptotic analysis so that all key features of its qualitative behavior are preserved.  \\
Let $\Omega\subset \mathbb{R}^2$ be a smooth bounded domain. Then, consider the following phase field PDE model of cell motility, studied in \cite{Ber14}:
\begin{equation}
\frac{\partial \rho_\ve}{\partial t}=\Delta \rho_\ve
-\frac{1}{\ve^2}W^{\prime}(\rho_\ve)
- P_\ve\cdot \nabla \rho_\ve +\lambda_\ve(t)\;\;
  \text{ in }\Omega,
\label{eq1}
\end{equation}
\begin{equation}
\frac{\partial P_\ve}{\partial t}=\ve\Delta P_\ve -\frac{1}{\ve}P_\ve
 -\beta \nabla \rho_\ve\;\;
 \text{ in } \Omega,
\label{eq2}
\end{equation}
where
\begin{equation*} \label{lagrange}
\lambda_\ve(t)=\frac{1}{|\Omega|}\int_\Omega\left(\frac{1}{\ve^2}W^\prime(\rho_\ve)
+ P_\ve\cdot \nabla \rho_\ve \right)\, dx,
\end{equation*}
is a {\em Lagrange multiplier} term responsible for total volume preservation of $\rho_\ve$, and
  \begin{equation}\label{doublewell}
  W(\rho)=\frac{1}{4} \rho^2 (1-\rho)^2
  \end{equation}
  is the Allen-Cahn (scalar Ginzburg-Landau) double equal well potential, and $\beta\geq 0$ is a physical parameter (see \cite{Zie12}). 
  \begin{figure}[H]
  \centering
  \includegraphics[width = .5\textwidth]{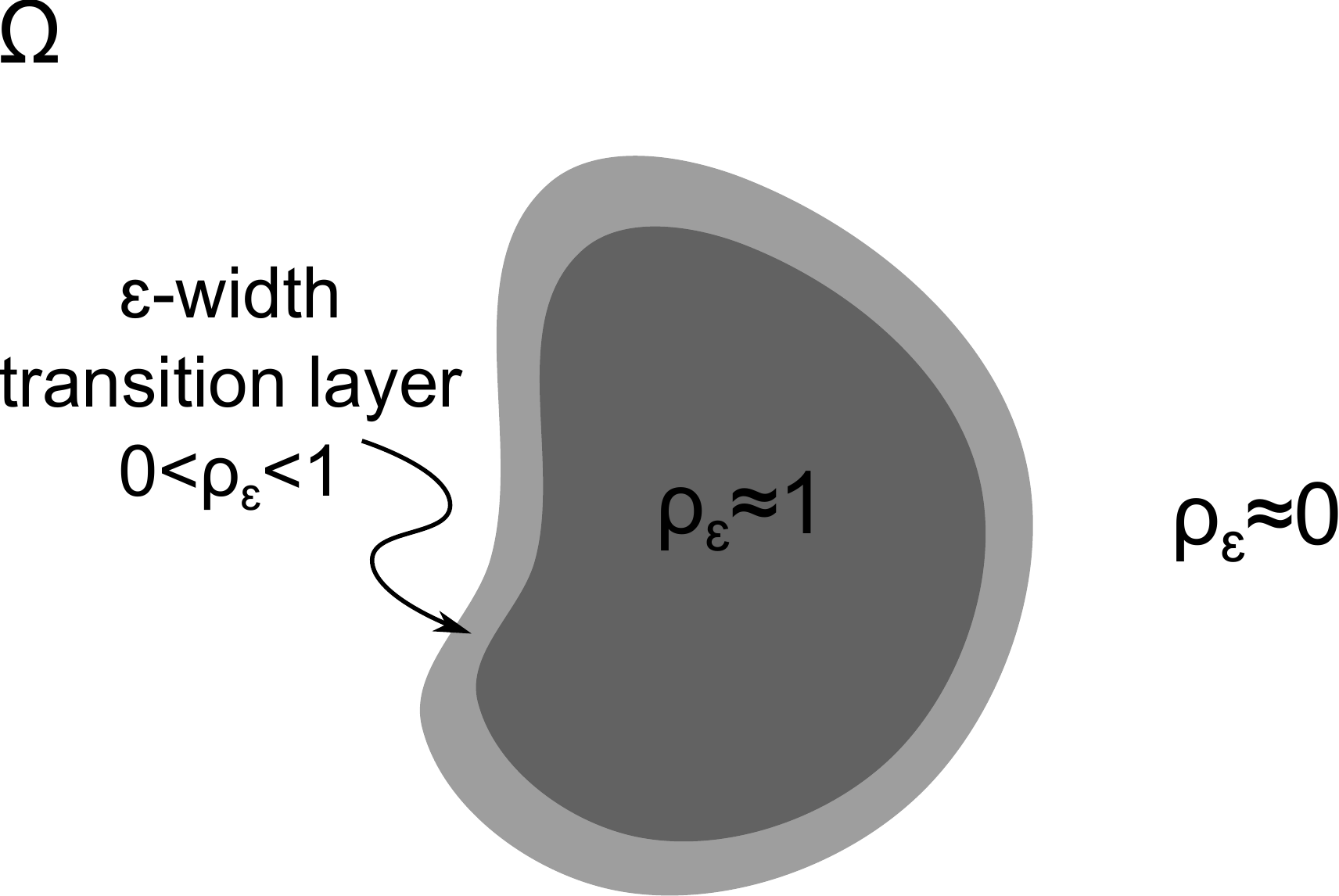}
  \caption{Sketch of phase-field parameter $\rho_\ve$}\label{figs:phasefield}
  \end{figure}
  In \eqref{eq1}-\eqref{eq2}, $\rho_\ve\colon \Omega\rightarrow \mathbb{R}$ is the phase-field parameter that, roughly speaking, takes values 1 and 0 inside and outside, respectively, a subdomain $D(t)\subset \Omega$ occupied by the moving cell. These regions are separated by a thin ``interface layer'' of width $O(\ve)$ around the boundary $\Gamma(t):=\partial D(t)$, where $\rho_\ve(x,t)$ sharply transitions from 1 to 0. The vector valued function $P_\ve \colon \Omega\rightarrow \mathbb{R}^2$ models the orientation vector due to polymerization of actin filaments inside the cell. On the boundary $\partial \Omega$ Neumann and Dirichlet boundary conditions respectively are imposed: $\partial_\nu \rho_\ve=0$ and   $P_\ve=0$.
  
 It was shown in \cite{Ber14} that  $\rho_\ve(x,t)$ converges to a characteristic function $\chi_{D(t)}$ as $\ve\to 0$, where $D(t)\subset \mathbb{R}^2$ . Namely, the phase-field parameter $0\leq \rho_\ve(x,t)\leq 1$ is equal to $1$ when $x\in D(t)$ and equal to $0$ outside of $D(t)$. This is referred to as the sharp interface limit of $\rho_\ve$ and we write $\Gamma(t):= \partial D(t)$. More precisely, given a closed non self-intersecting curve $\Gamma(0)\subset \mathbb{R}^2$, consider the initial profile
  \begin{equation*}\label{phaseIC}
  \rho_\ve(x,0) = \theta_0\left(\frac{\operatorname{dist}(x,\Gamma(0))}{\ve}\right)
  \end{equation*}
  where 
  \begin{equation*}\label{steadysolution}
  \theta_0(z) = \frac{1}{2}\left(\tanh\left(\frac{z}{2\sqrt{2}}\right)+1\right)
  \end{equation*}
  is the standing wave solution of the Allen-Cahn equation, and $\operatorname{dist}(x,\Gamma(0))$ is the signed distance from the point $x$ to the curve $\Gamma(0)$. Then, $\rho_\ve(x,t)$ has the asymptotic form
  \begin{equation*}\label{asymptoticform}
  \rho_\ve(x,t) = \theta_0\left(\frac{\operatorname{dist}(x,\Gamma(t))}{\ve}\right)+ O(\ve),
  \end{equation*}
  where $\Gamma(t)$ is a curve which describes the boundary of the cell. \newline
  
  It was formally shown in \cite{Ber14} that the curves $\Gamma(t)$ obey the evolution equation
  \begin{equation}\label{interface0}
  V(s,t) = \kappa(s,t) + \Phi(V(s,t)) - \frac{1}{|\Gamma(t)|}\int_\Gamma \kappa(s',t) +\Phi(V(s',t)) ds',
  \end{equation}
  where $s$ is the arc length parametrization of the curve $\Gamma(t)$, $V(s,t)$ is the  normal velocity of curve $\Gamma(t)$ w.r.t. inward normal at location $s$, $|\Gamma(t)|$ is the length of $\Gamma(t)$, $\kappa(s,t)$ is the signed curvature of $\Gamma(t)$ at location $s$, and $\Phi(\cdot)$ is a known smooth, non-linear function. \newline
 
  \begin{remark}\label{rem:betaequiv}
  In \cite{Ber14} it was shown that $\Phi(V)=\beta\Phi_0(V)$ 
  where $\Phi_0(V)$ is given by the equation
  \begin{equation}\label{phi}
  \Phi_0(V) := \int_\mathbb{R} \psi(z; V)(\theta_0')^2 dz
  \end{equation}
  and $\psi(z)=\psi(z;V)$ is the unique solution of
  \begin{equation}\label{phi1}
  \psi''(z)+V\psi'(z)-\psi(z)+ \theta_0' =0,
  \end{equation}
  with $\psi(\pm \infty)=0$. 
  
  \end{remark}
 
  The case $\beta=0$ in equations \eqref{eq1}-\eqref{eq2} leads to $\Phi\equiv 0$ in \eqref{interface0}, thus reducing to a mass preserving analogue of the Allen-Cahn equation. Properties of this equation were studied in \cite{Che10,Gol94}, and it was shown that the sharp interface limit as $\ve\to 0$ recovers volume preserving mean curvature motion: $V=\kappa-\frac{1}{|\Gamma|}\int_{\Gamma} \kappa ds.$

  Equations \eqref{eq1}-\eqref{eq2} are a singularly perturbed  parabolic PDE system in two spatial dimensions. Its sharp interface limit given by \eqref{interface0} describes evolution of the   curve $\Gamma(t)$ (the sharp interface). Since  $V(s,t)$ and  $\kappa(s,t)$ are expressed via first and second derivatives of  $\Gamma(t)(=\Gamma(s,t))$, equation \eqref{interface0} can be viewed as the second order PDE for  $\Gamma(s,t)$.   
  Since this PDE has spatial dimension one and it  does not contain a singular perturbation, qualitative and numerical analysis of \eqref{interface0} is much simpler  than  that  of the  system \eqref{eq1}-\eqref{eq2}. 
  
  
\begin{remark}  
It was observed in \cite{Ber14} that both the analysis and the behavior of  solutions of system \eqref{eq1}-\eqref{eq2} crucially depends on the parameter $\beta$. Specifically there is critical value $\beta_{cr}$ such that for $\beta>\beta_{cr}$ complicated phenomena of non-uniqueness and hysteresis arise. This critical value is defined as the maximum $\beta$ for which $V-\beta\Phi_0(V)$ is a monotone function of $V$.

 While this supercritical regime is a subject of the ongoing investigation, in this work focus on providing a rigorous analysis of subcritical regime $\beta<\beta_{cr}$. For equation \eqref{interface0}     
the latter regime corresponds to the case of {\it monotone function} 
$V-\Phi(V)$. 
\end{remark}

%
 \subsection{Biological background: cell motility problem}
 
 In \cite{Zie12} a phase field model that describes crawling motion of keratocyte cells on substrates was introduced.  Keratocyte cells are typically harvested from scales of fish (e.g., cichlids \cite{Ker08})  for {\em in vitro} experiments. Additionally, humans have keratocyte cells in their corneas. These cells are crucial during wound healing, e.g., after corrective laser surgery \cite{Moh03}.

  The biological mechanisms which give rise to keratocyte cell motion are complicated and they are an ongoing source of research. Assuming that a directional preference has been established, a keratocyte cell has the ability to maintain self-propagating cell motion via internal forces generated by a protein {\em actin}.  Actin monomers are polarized in such a way that several molecules may join and form filaments.  These actin filaments form a dense and rigid network at the leading edge of the cell within the cytoskeleton, known as the lamellipod.  The lamellipod forms adhesions to the substrate and by a mechanism known as {\em actin tread milling} the cell protrudes via formation of new actin filaments at the leading edge.  
  
  We may now explain the heuristic idea behind the model. Roughly speaking, cell motility is determined by two (competing) mechanisms: surface tension and  protrusion due to actin polymerization. 
  The domain where $\rho_\ve(x)\approx 1$  is occupied  the cell and $P_\ve$ as the local averaged orientation of the filament network.  Surface tension enters the model \eqref{eq1}-\eqref{eq2} via the celebrated Allen-Cahn equation with double-well potential \eqref{doublewell}:
  \begin{equation*}\label{ACmotion}
  \frac{\partial \rho_\ve}{\partial t} = \Delta \rho_\ve -\frac{1}{\ve^2} W'(\rho_\ve).
  \end{equation*}
   In the sharp interface limit ($\ve\to 0$), surface tension  leads to  the  curvature driven motion of the interface. The actin polymerization enters the system \eqref{eq1}-\eqref{eq2} through the $-P_\ve \cdot \nabla \rho_\ve$ term. Indeed, recall
   \begin{equation*}\label{materialderivative}
  \frac{D \rho_\ve}{Dt} = \frac{\partial \rho_\ve}{\partial t} + P_\ve \cdot \nabla \rho_\ve
  \end{equation*}
  as the material derivative of $\rho_\ve$ subject to the velocity field $P_\ve$. Thus the term $-P_\ve \cdot \nabla \rho_\ve$ is an advective term generated by actin polymerization. \newline
  The last term of \eqref{eq1}, $\lambda_\ve(t)$ is responsible for volume preservation, which is an important physical condition.
  The diffusion term $\ve \Delta P_\ve$ corresponds to diffusion of actin and does not significantly affect the dynamics of $\rho_\ve$. 
  The term $-\beta\nabla \rho_\ve$ describes the creation of actin by polymerization, which leads to a protrusion force. It gives the rate of growth of polymerization of actin: $\frac{\partial P_\ve}{\partial t} \sim -\beta \nabla \rho_\ve$.   The $\frac{1}{\ve} P_\ve$  term provides decay of $P_\ve$ away from the interface, for example due to depolymerization. 
  
 The system \eqref{eq1}-\eqref{eq2} is a slightly modified form of the model proposed in \cite{Zie12}. 
 It preserves  key features of the qualitative behavior yet is more convenient for  mathematical analysis.
  

    \subsection{Overview of results and techniques}\label{sec:overview}
    A main goal of this work is to prove existence of a family of curves which evolve according to the  equation \eqref{interface0} (that describes evolution of the sharp interface) and investigate their properties. The problem of mean curvature type motion was extensively studied by mathematicians from both PDE and geometry communities for several decades. A review of results on unconstrained motion by mean curvature can be found \cite{Bra78,GagHam86,Gra87}. Furthermore the viscosity solutions techniques have been efficiently applied in the PDE analysis of such problems. These techniques do not apply to mean curvature motion with volume preservation constraints \cite{Gag86,Che93,Ell97}, and the analysis becomes
    especially difficult in dimensions greater than two \cite{Esc98}. Note that existence in two dimensional mean curvature type motions were recently studied (e.g., \cite{Bon00,Che93, Ell97}) and appropriate techniques of regularization were developed. Recently, analogous issues resurfaced in the novel context of biological cell motility problems, after a phase-field model was introduced in \cite{Zie12}.

The problem studied in the present work is two dimensional (motion of a curve on the plane). The distinguished feature of this problem is that the velocity enters the evolution equation implicitly via a {\em non-linear} function $V-\Phi(V)$. Therefore the time derivative in the corresponding PDE that describes the signed distance, $u(\sigma,t)$ of the curve from a given reference curve also enters the PDE implicitly, which leads to challenges in establishing existence.

	The following outlines the basic steps of the analysis. 
	First, we consider smooth ($H^2$) initial data and generalize the regularization idea from \cite{Che93} (see also \cite{Ell97}) for the implicit dependence described above. Here, the main difficulty is to establish existence on a time interval that does not depend on the regularization parameter $\ve$. To this end, we derive $L^2$ (in time and space) a priori estimates independent of $\ve$ for third order derivatives and uniform in time $L^2$ estimates for  second order derivatives. These estimates allow us to show existence and to pass to the limit as $\ve\to 0$, and they are derived by considering the equation for $u_\sigma=\frac{\partial u}{\partial \sigma}$. We use ``nice'' properties of this equation to obtain higher order estimates independent of $\ve$ (which are not readily available for the equation for $u$).  In particular, it turns out that the equation for $u_\sigma$  can be written as a quasi linear parabolic PDE in divergence form.   For such equations   quite remarkable classical results  establish  Holder continuity of solutions for even for discontinuous initial data  \cite{Lad68}. This  provides a lower bound on the  possible blow up time, which does depend on $H^2$ norm of initial data for $u$ in our problem.
	
	 As a result, we establish existence on a time interval that depends on the $H^2$ norm of initial data.
	
	Second, observe that experiments for cell motility show that the cell shape is not necessarily smooth. Therefore one needs to consider more realistic models where smoothness assumptions on initial conditions are relaxed. 
	In particular,
	one should allow for corner-type singularities. To this end, we pass to generic $W^{1,\infty}$ initial curves. For the limiting equations for $u$ and $u_\sigma$, we show existence on a time interval that does not depend on the $H^2$ (and $H^1$ for $u_\sigma$) norms of initial conditions. This is necessary because these norms blow up for non-smooth functions from $W^{1,\infty}\setminus H^2$. The existence is proved by a maximum principle type argument, which is not available for the regularized equations that contain fourth order derivatives. Also it is crucial to establish H\"older continuity results for $u_\sigma$, rewriting the equation for $u_\sigma$ as a quasilinear divergence form parabolic PDE.  

	

  

  
  After proving short time existence we address the issues of global existence of such curves. The latter is important for the comparison of theoretical and experimental results on cell motility. We will present an exploratory study which combines analytical and numerical results for the long time behavior of the cell motility model. Analytically, we prove that similarly to the classical curvature driven motion with volume preservation, traveling waves with nonzero velocity do not exist. While through numerical experiments, we observe a nontrivial (transient) net motion resulting from the non-linearity and asymmetry of the initial shape. This observation shows an essential difference from the classical area preserving curvature driven motion.
  
	
Numerically solving \eqref{interface0} is a nontrivial task due to the non-linearity and non-locality in the formulation of the normal velocity. Classical methods such as level-set methods cannot be readily used here. 
We introduce an efficient algorithm which separates the difficulties of non-linearity and non-locality and resolves them independently through an iterative scheme. The accuracy of the algorithm is validated by numerical convergence. Our numerical experiments show that both non-linearity and asymmetry of the initial shape (in the form of non-convexity) can introduce a drift of the center of mass of the cell. Increasing the effects of the non-linearity or increasing the asymmetry results in an increase in the drift distance. 
  \textcolor{blue}{}

\section{Existence of solutions to the evolution equation (\ref{interface0})}
\label{sec:existence}

We study curves propagating via the evolution equation \eqref{interface0}.
The case $\Phi\equiv 0$ corresponds to well-studied volume preserving curvature motion (see, e.g.,  \cite{Che93, Ell97,Esc98,Gag86}).  
We emphasize that the presence of $\Phi(V)$ results in an implicit, non-linear and non-local equation for $V$, which leads to challenges from an analytical and numerical standpoint. \newline
The goal of this section is to prove the following:
	\begin{theorem}\label{thm1}
		Let $\Phi\in L^\infty(\mathbb{R})$ be a Lipschitz function satisfying
		\begin{equation}\label{phibound}
		\|\Phi'\|_{L^\infty(\mathbb{R})} < 1.
		\end{equation}
		Then, given $\Gamma_0\in W^{1,\infty}$, a closed and non self-intersecting curve on $\mathbb{R}^2$, there is a time $T=T(\Gamma_0)>0$ such that a family of curves $\Gamma(t)\in H^2$  exists for $t\in (0,T]$ which satisfies the evolution equation \eqref{interface0} with initial condition $\Gamma(0)=\Gamma_0$.
	\end{theorem}
	
	\begin{remark}
		The classes $W^{1,\infty}$ and $H^2$ above refer to curves which are parametrized by mappings from Sobolev spaces $W^{1,\infty}$ and $H^2$ correspondingly.
	\end{remark}
	
	\begin{remark}
		After time $T$ the curve could self-intersect or blow-up in the parametrization map (e.g., a cusp) could occur.
	\end{remark}

		We first prove the existence for smooth ($H^2$) initial data. The main effort is to pass to non-smooth initial conditions (e.g., initial curves with corners), see discussion in Section \ref{sec:overview}.

	
	
%

\noindent{Proof of Theorem \ref{thm1}:} The proof of Theorem \ref{thm1} is split into 4 steps. In Step 1 we present a PDE formulation of the evolution problem \eqref{interface0} and introduce its regularization by adding a higher order term with the small parameter $\ve>0$. 
In Step 2 we prove a local in time existence result for the regularized problem.  In Step 3  we establish a uniform time interval of existence for solutions of the regularized problem via a priori estimates. These estimates allow 
us to pass to the limit $\ve\to 0$, which leads to
existence for \eqref{interface0} for smooth initial data.
Finally, Step 4 is devoted to the transition from $H^2$-smooth initial data to $W^{1,\infty}$ ones. A crucial role here plays derivation of $L^\infty$ bounds for the solution and its first derivative independent of $H^2$ norm of the initial data.  

	
	\noindent {\bf Step 1.} {\em Parametrization and PDE forms of \eqref{interface0}}. 
	
Let $\tilde\Gamma_0$ be a $C^4$ smooth reference curve in a small neighborhood of $\Gamma_0$ and   
let $\tilde\Gamma_0$ be parametrized by arc length parameter $\sigma\in I$.  Let $\kappa_0(\sigma)$ be the signed curvature of $\tilde\Gamma_0$ and $\nu(\sigma)$ be the the inward pointing normal vector to $\tilde\Gamma_0$. Consider the tubular 
neighborhood
\begin{equation*}\label{existball}
U_{\delta_0} := \{x \in \mathbb{R}^2 \mid {\rm dist}(x,\tilde{\Gamma}_0)<2\delta_0\}.
\end{equation*}

	\begin{figure}[h]
	\includegraphics[width = .8\textwidth]{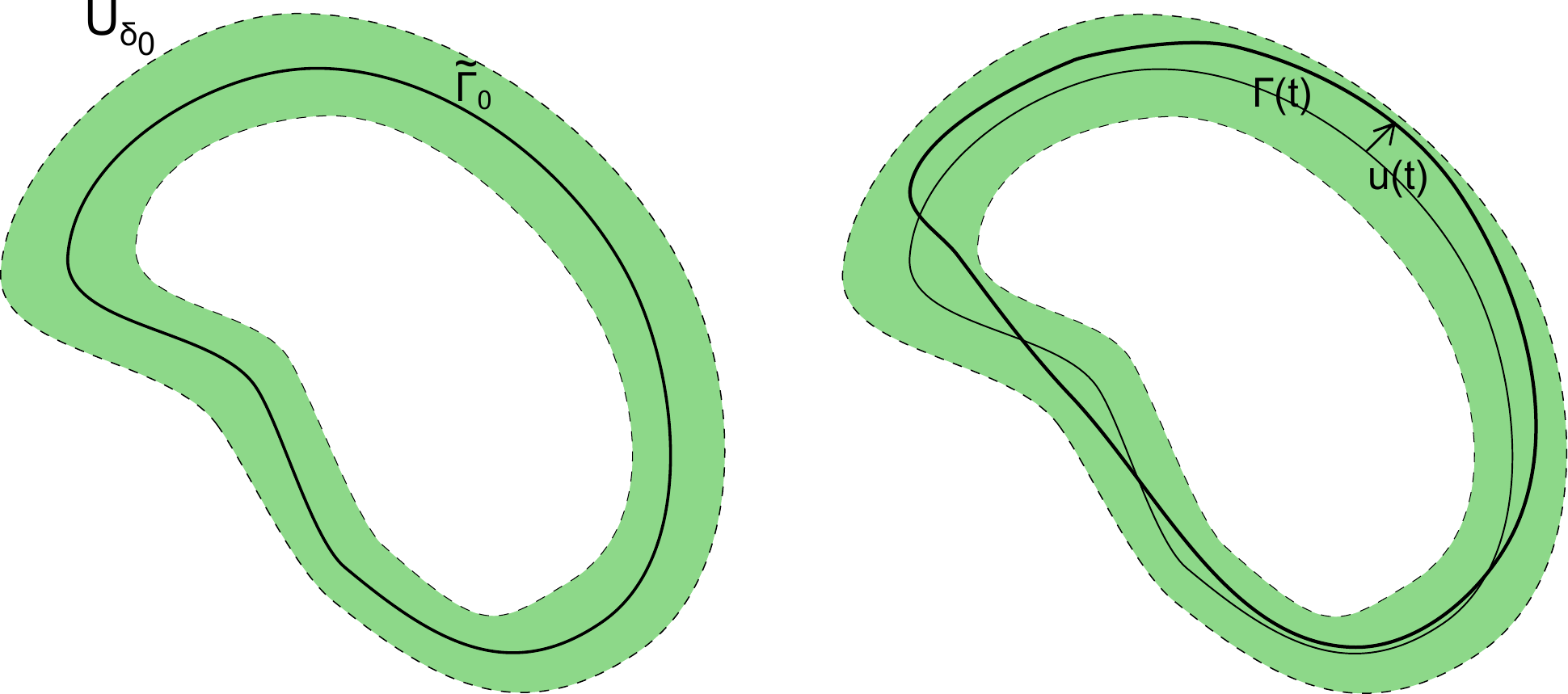}
	\caption{Visualization of $\mathbb{B}_{\delta_0}$ and the relation between $d(t)$ and $\Gamma(t)$}
	\label{fig:existenceball}
	\end{figure}
One can choose $\tilde \Gamma_0$  and  $\delta_0$ 
so that the map
\[
Y\colon I\times (-2\delta_0,2\delta_0) \rightarrow U_{\delta_0}, \; Y(\sigma,u) := \tilde\Gamma_0(\sigma)+u \nu(\sigma)
\]
is a $C^2$ diffeomorphism between $I\times (-2\delta_0,2\delta_0)$ and its image, and  $\Gamma_0\subset Y(I\times (-2\delta_0,2\delta_0))$  . 
Then $\Gamma_0$ can be parametrized by $\Gamma_0=\tilde \Gamma_0(\sigma)+u_0(\sigma)\nu(\sigma)$,
$\sigma\in I$, for some periodic function $u_0(\sigma)$. 
Finally, we can assume that $\delta_0$ is sufficiently small so that 
\begin{equation}\label{smalldelta}
\delta_0\|\kappa_0\|_{L^\infty} < 1,
\end{equation}
where $\kappa_0$ denotes the curvature of $\tilde \Gamma$. 

A continuous function  $u: I\times [0,T]\rightarrow [-\delta_0,\delta_0]$, periodic in the $\sigma$ variable, describes a family of closed curves via the mapping
\begin{equation}\label{equivalenceofD}
\Gamma(\sigma,t) =\tilde \Gamma_0(\sigma)+u(\sigma,t)\nu(\sigma).
\end{equation}
That is, there is a well-defined correspondence between $\Gamma(\sigma,t)$ and $u(\sigma,t)$.



Recall the Frenet-Serre formulas applied to $\tilde\Gamma_0$:
\begin{equation}\label{dtau}
\frac{d\tau}{d\sigma}(\sigma)= \kappa_0(\sigma) \nu(\sigma)
\end{equation}
\begin{equation}\label{dn}
\frac{d\nu}{d\sigma}(\sigma) =-\kappa_0(\sigma)\tau(\sigma)
\end{equation}
where $\tau$ is the unit tangent vector. Using \eqref{equivalenceofD}-\eqref{dn} we express the normal velocity 
$V$ of $\Gamma$ as
\begin{equation}\label{veqn}
V
=  \frac{1-u\kappa_0}{S}u_t
\end{equation}
where $S=S(u)= \sqrt{u^2_\sigma + (1-u\kappa_0)^2}$. 
Also, in terms of $u$, curvature of $\Gamma$ is given by 
\begin{equation}\label{curvature}
\kappa(u) = 
\frac{1}{S^3} \Bigl((1-u \kappa_0) u_{\sigma\sigma} + 2\kappa_0 u_\sigma^2+(\kappa_0)_\sigma u_\sigma u  + 
\kappa_0(1-u \kappa_0)^2\Bigr).
\end{equation}
Note that if $u$ is sufficiently smooth \cite{Doc76} one has
$$
\int_I \kappa(u) Sd\sigma=\int_{\Gamma} \kappa ds =2\pi,
$$
in particular this holds for every $u\in H^2_{per}(I)$ such that $|u|\leq \delta_0$ on $I$. 
Hereafter, $H_{per}^k(I)$  denote the Sobolev spaces of 
periodic functions on $I$ with square-integrable 
derivatives up to the $k$-th order and $W^{1,\infty}_{per}(I)$ denotes the space of periodic functions with the first derivative 
in $L^\infty(I)$.      



 Combining \eqref{veqn} and \eqref{curvature}, we rewrite \eqref{interface0} as the following PDE for $u$:
 \begin{equation}\label{deqn}
 \begin{aligned}
u_t -\frac{S}{1-u \kappa_0}\Phi\left(\frac{1-u\kappa_0}{S}u_t\right)
 &= \frac{S}{1-u \kappa_0}\kappa(u)\\
&-\frac{S}{(1-u \kappa_0)L[u]}\Bigl( \int_I \Phi\left(\frac{1-u \kappa_0}{S} u_t\right)Sd\sigma+2\pi\Bigr),
\end{aligned}
\end{equation}
where 
$$
L[u]=\int_I S(u)d\sigma
$$
is the total length of the curve.

The initial condition $u(\sigma,0)=u_0(\sigma)$ corresponds to the initial profile $\Gamma(0)=\Gamma_0$. Since 
(\ref{deqn}) is not resolved with respect to the time derivative $u_t$, it is natural to resolve \eqref{interface0} with respect to $V$
to convert (\ref{deqn}) into a parabolic type PDE where the time derivative $u_t$ is expressed as a function of $u$, $u_\sigma$, $u_{\sigma\sigma}$. 
The following lemma shows how to rewrite \eqref{deqn} in such a form. This is done by resolving \eqref{interface0} with respect to $V$ to get a Lipschitz  continuous resolving  map, provided that $\Phi \in L^\infty(\mathbb{R})$ and $\|\Phi^\prime\|_{L^\infty(\mathbb{R})}<1$. 

It is convenient to rewrite in the form \eqref{interface0}
\begin{equation}\label{lameq1}
V = \kappa(u)+ \Phi(V) -\lambda,\quad \int_I V S(u)d\sigma=0,  
\end{equation}
where both normal velocity $V$  and constant $\lambda$  are considered as unknowns.

\begin{lemma}
\label{lem1} Suppose that $\Phi\in L^\infty(\mathbb{R})$ and $\|\Phi' \|_{L^\infty(\mathbb{R})}<1$. Then for any $u (\sigma)~\in H^2_{per}(I)$, such that $|u(\sigma)|\leq \delta_0$, 
there exists a unique  solution $(V(\sigma), \lambda)\in L^2(I)\times \mathbb{R}$  of \eqref{lameq1}.  
Moreover, the resolving map $\mathcal{F}$ assigning to a given  $u\in H^2_{per}(I)\cap \{u; \,|u|\leq\delta_0\}$ the solution 
$V=\mathcal{F}(u)\in L^2(I)$ of \eqref{lameq1} is locally Lipschitz continuous. 
\end{lemma}	
\begin{proof} 
Let $J:=\|\Phi' \|_{L^\infty(\mathbb{R})}$. Fix $\kappa\in C_{per}(I)$, a positive function $S\in C_{per}(I)$ and $\lambda\in \mathbb{R}$, and consider the equation
\begin{equation}\label{lameq}
	V = \kappa+ \Phi(V) -\lambda.
\end{equation}
It is immediate that the unique solution of \eqref{lameq} is given by $V=\Psi(\kappa-\lambda)$, 
where $\Psi$ is the inverse map to $V-\Phi(V)$. Note that $\frac{1}{1+J}\leq \Psi^\prime\leq \frac{1}{1-J}$, 
therefore $\Psi$ is strictly increasing function and $\Psi(\kappa-\lambda)\to\pm\infty$ uniformly as $\lambda\to\mp \infty$. 
It follows that there exists a unique $\lambda\in \mathbb{R}$ such that $V=\Psi(\kappa-\lambda)$ is a solution 
of \eqref{lameq} satisfying 
\begin{equation}
\int_I V S d\sigma=0.
\label{DopCondLambda}
\end{equation}
Next we establish the Lipschitz continuity of the resolving map $(\kappa,S)\mapsto V\in L^2(I)$ 
as a function of $\kappa\in L^2(I)$ 
and $S\in L^\infty_{per}(I)$, $S\geq 1-\delta_0\|\kappa_0\|_{L^\infty(I)}>0$ (cf. \eqref{smalldelta}),
still assuming that $\kappa,S\in C_{per}(I)$.  
Multiply \eqref{lameq} by $VS$ and integrate over $I$ to find with the help of the Cauchy-Schwarz inequality,
\begin{equation*}
\begin{aligned}
\int_I V^2 Sd\sigma 
&=  \int_I \kappa V Sd\sigma + \int_I \Phi(V) V Sd\sigma \\
&\leq \|S\|_{L^\infty(I)}\bigl(\|\kappa\|_{L^2(I)}+(|I|\|\Phi\|_{L^\infty(\mathbb{R})})^{1/2}\bigr)\Bigl( \int_I V^2 Sd\sigma\Bigr)^{1/2}.
\end{aligned}
\end{equation*}
Recalling that $S \geq \omega:= 1-\delta_0\|\kappa_0\|_{L^\infty(I)}>0$, we then obtain
\begin{equation}\label{Vbounded}
\|V\|_{L^2(I)} \leq \frac{\|S\|_{L^\infty(I)}}{\omega^{1/2}}\bigl( \|\kappa\|_{L^2(I)}+(|I|\|\Phi\|_{L^\infty(\mathbb{R})})^{1/2} \bigr).
\end{equation}
To see that $\kappa\mapsto V$ (for fixed $S$) is Lipschitz continuous, consider solutions $V_1$, $V_2$ of \eqref{lameq}-\eqref{DopCondLambda}
with $\kappa=\kappa_1$ and $\kappa=\kappa_2$. Subtract the equation for $V_2$ from that for $V_1$ multiply by $S(V_1-V_2)$
and integrate over $I$,
$$
\int_I \bigl( 
(V_1-V_2)^2-(\Phi(V_1)-\Phi(V_2))(V_1-V_2)\bigr) Sd\sigma =  \int_I (\kappa_1-\kappa_2) (V_1-V_2) Sd\sigma.
$$  
Then, since $|\Phi(V_1)-\Phi(V_2)|\leq J |V_1-V_2|$ we derive that 
$$
\|V_1-V_2\|_{L^2(I)}\leq 
\frac{\|S\|_{L^\infty(I)}}{\omega^{1/2}(1-J)} \|\kappa_1-\kappa_2\|_{L^2(I)}.
$$
Next consider solutions of \eqref{lameq}-\eqref{DopCondLambda}, still denoted  $V_1$  and $V_2$, which correspond now
$S=S_1$ and $S=S_2$ with the same $\kappa$.  Subtract the equation for $V_2$ from that for $V_1$
multiply by $S_1(V_1-V_2)+V_2(S_1-S_2)$ and  integrate over $I$ to find 
$$
\begin{aligned}
\int_I \bigl( 
(V_1-V_2)^2-&(\Phi(V_1)-\Phi(V_2))(V_1-V_2)\bigr) S_1 d\sigma \\&= 
 \int_I V_2(S_1-S_2)\bigl(\Phi(V_1)-\Phi(V_2) -(V_1-V_2)\Bigr) d\sigma.
\end{aligned}
$$  
Then applying the Cauchy-Schwarz inequality we derive
\begin{equation*}
 \|V_1-V_2\|_{L^2(I)}\leq \frac{1+J}{\omega (1-J)}
\|V_2\|_{L^2(I)}\|S_1-S_2\|_{L^\infty(I)}.
\end{equation*}
Thanks to \eqref{Vbounded} this completes the proof of local Lipschitz continuity of the resolving 
map on the dense subset (of continuous functions) in 
$$\Theta=\Bigl\{(\kappa,S)\in L^2(I)\times \{S\in L^\infty_{per}(I); S\geq \omega\}\Bigr\}$$ 
and thus on the whole set $\Theta$. It remains to note that the map $u\mapsto (\kappa(u), S(u))$ 
is locally Lipschitz on $\{u\in H^2_{per}(I);\, |u|\leq \delta_0\ \text{on}\ I\}$, which 
completes the proof of the Lemma.     
%
%
%
%
\end{proof}
\begin{remark}The parameter $\lambda\in \mathbb{R}$ with the property that the solution $V$ of \eqref{lameq} satisfies $\int_I V S d\sigma =0$ is easily seen to be
	\begin{equation}
	\lambda = \frac{1}{L[u]} \int_I (\kappa(u)+\Phi(V)S d\sigma= \frac{\pi}{L[u]}+\frac{1}{L[u]}\int_I \Phi(V) S d\sigma.
	\end{equation}
\end{remark}
\begin{remark} Under conditions of Lemma \ref{lem1}, if $\kappa(u)\in H_{per}^1(I)$ ($u\in H^3_{per}(I)$) then it holds that $V=\mathcal{F}(u)\in H_{per}^1(I)$.
\end{remark}


Equation \eqref{interface0} is equivalently rewritten in terms of the resolving operator $\mathcal{F}$ as
\begin{equation}\label{deqn1}
u_t = \frac{S(u)} {1- u \kappa_0}\mathcal{F}(u),\ \text{or}\  u_t = \mathcal{\tilde F}(u),
\end{equation}
where $\mathcal{\tilde F}(u):=S(u)\mathcal{F}(u)/(1- u \kappa_0)$.


\noindent {\bf Step 2.} {\em Introduction and analysis of regularized PDE.}

We now introduce a small parameter regularization term to $\eqref{deqn}$ which allows us to apply standard existence results.
To this end, let $u^\ve=u^\ve(\sigma,t)$ solve the following regularization of equation \eqref{deqn1}
for $0<\ve\leq1$, 
\begin{equation}\label{evoMotion}
u^\ve_t + 
\epsilon u^\ve_{\sigma\sigma\sigma\sigma} = 
\mathcal{\tilde F}(u^\ve),
\end{equation}
with 
 $u^\ve(\sigma,0)=u_0$. Define
\begin{equation*}\label{regVelocity}
V^\ve := \frac{1-u^\ve\kappa_0}{S^\ve}\left(u_t^\ve + \ve u_{\sigma\sigma\sigma\sigma}^\ve\right),
\end{equation*}
where $S^\ve=S(u^\ve)$.
Since $V^\ve = \mathcal{F}(u^\ve)$, then by definition of the resolving map $\mathcal{F}$ we have that $u^\ve$ satisfies the following equation:
\begin{equation}\label{eqnfordve}
\begin{aligned}
u_t^\ve +\ve u_{\sigma\sigma\sigma\sigma}^\ve -
\frac{S^\ve}{1-u \kappa_0}\Phi(V^\ve)
 &= \frac{S^\ve}{1-u^\ve \kappa_0}\kappa(u^\ve)\\
&-\frac{S^\ve}{(1-u^\ve \kappa_0)L[u^\ve]}
\Bigl( \int_I \Phi(V^\ve)S^\ve d\sigma+2\pi\Bigr).
\end{aligned}
\end{equation}

Hereafter we consider $H^2_{per}(I)$ equipped with the norm 
\begin{equation*}
\|u\|^2_{H^2(I)}:= \|u\|_{L^2(I)}^2+\|u_{\sigma\sigma}\|_{L^2(I)}^2.
\end{equation*}
\begin{proposition}\label{prop1}
Let $\Gamma_0$ and $\Phi(\cdot)$ satisfy the conditions of Theorem \ref{thm1}. 
Assume that $u_0\in H^2_{per}$ and $\max|u_0|<\delta_0$. 
Then there exists a non-empty interval $[0,T^\ve]$ such that a 
solution $u^\ve$ of \eqref{eqnfordve} with initial data 
$u^\ve(\eta,0)=u_0(\eta)$
exists and 
\begin{equation}
\label{regularity}
\begin{aligned}
 u^\ve \in L^2(0,T^\ve; H^4_{per}(I)) &\cap H^1(0,T^\ve;L^2(I))
\cap L^\infty(0,T^\ve; H^2_{per})\\ &\text{and}\ \sup_{t\in[0,T]} \|u^\ve(t)\|_{L^\infty(I)}\leq \delta_0.
\end{aligned}
\end{equation}
Furthermore, 
this solution can be extended on a bigger time interval $[0,T^\ve+\Delta_t]$
so long as $\delta_0-\max u(\sigma,T^\ve)\geq \alpha$   and $\|u(T^\ve)\|_{H^2(I)}\leq M$
for some $\alpha>0$ and $M<\infty$, where $\Delta_t$ depend on $\ve$, $\alpha$ and $M$, $\Delta_t:=\Delta_t(\alpha, M,\ve)>0$.
\end{proposition}

\begin{proof} Choose $T^\ve>0$ and $M>\|u_0\|_{H^2(I)}$, 
and introduce the set 
$$
K:=\{u: \|u\|_{H^2}\leq M,\, |u|\leq \delta_0\, \text{in} \ I\}.
$$
Given $\tilde{u}\in L^\infty(0,T^\ve;K)$, consider the following auxiliary problem 
\begin{equation}\label{linear}
u_t + \ve u_{\sigma\sigma\sigma\sigma}= 
\mathcal{\tilde F}(\tilde{u})
\end{equation}
with $u(\sigma,0)=u_0(\sigma)$. 
Classical results (e.g., \cite{Lio72}) yield existence of a unique solution 
$u$ of \eqref{linear} which possesses the following regularity
\begin{equation*}
u\in L^2(0,T^\ve; H^4_{per}(I)) \cap H^1(0,T^\ve;L^2(I))\cap L^\infty(0,T^\ve;H^2_{per}(I)).
\end{equation*}
That is, a resolving operator 
\begin{equation*} 
\mathcal{T}\colon L^\infty(0,T^\ve;K)\to L^\infty(0,T^\ve;H^2(I))
\end{equation*}
which maps $\tilde{u}$ to the solution $u$ is well defined. Next we show that $\mathcal{T}$ is a contraction in $K$, 
provided that  $T^\ve$ is chosen sufficiently small.

Consider  $\tilde{u}_1,\tilde{u}_2\in  L^\infty(0,T^\ve; K)$ satisfying the initial condition 
$\tilde{u}_1(\sigma,0)=\tilde{u}_2(\sigma,0)=u_0(\sigma)$ and define $u_1:=\mathcal{T}(\tilde{u}_1)$, 
$u_2:=\mathcal{T}(\tilde{u}_2)$. Let $\bar{u}:=u_1-u_2$. Then multiply the equality
$
\bar{u}_t + \ve \bar{u}_{\sigma\sigma\sigma\sigma} = \mathcal{\tilde F}(\tilde{u}_1)-\mathcal{\tilde F}(\tilde{u}_2)
$
by $(\bar{u}_{\sigma\sigma\sigma\sigma}+\bar{u})$ and integrate. 
Integrating by parts and using the Cauchy-Schwarz inequality  yields
\begin{align*}
\frac{1}{2}\frac{d}{dt}\int_I (\bar{u}_{\sigma\sigma}^2+ \bar{u}^2)d\sigma&+\ve 
\|\bar{u}_{\sigma\sigma\sigma\sigma}\|^2_{L^2(I)}+\ve \|\bar{u}_{\sigma\sigma}\|^2_{L^2(I)} \\
&\leq \|\mathcal{\tilde F}(\tilde u_1)-\mathcal{\tilde F}(\tilde u_2)\|_{L^2(I)}(\|\bar{u}_{\sigma\sigma\sigma\sigma}\|_{L^2(I)}+\|\bar{u}\|_{L^2(I)}).
\end{align*}
Note that by Lemma \ref{lem1} the map  $\mathcal{F}(u)$ with values in $L^2(I)$ is 
Lipschitz on $K$; since $\mathcal{\tilde F}(u)=S(u)\mathcal{F}(u)/(1- u \kappa_0)$ it is not hard to 
see that $\mathcal{\tilde F}(u)$ is also Lipschitz. 
Using this and applying Young's inequality to the right hand side we obtain
 \begin{equation}
 \label{contproof1} 
 \frac{1}{2}\frac{d}{dt}\|\bar{u}\|^2_{H^2(I)}\leq C\|\tilde{u}_1-\tilde{u}_2\|_{H^2(I)}^2+\frac{1}{2}\|\bar{u}\|_{H^2(I)}^2
 \end{equation}
with a constant $C$ independent of $\tilde{u}_1$ and $\tilde{u}_2$ and $T_\ve$. 
Applying the Gr\"onwall inequality on \eqref{contproof1} we get
\begin{equation}
\label{contend}
\sup_{0\leq t\leq T^\ve} \|\mathcal{T}(\tilde{u}_1)-\mathcal{T}(\tilde{u}_2)\|^2_{H^2(I)} \leq 2(e^{T^\ve}-1)C\|\tilde{u}_1-\tilde{u}_2\|_{L^\infty(0,T^\ve;H^2(I))}
\end{equation}
Similar arguments additionally yield the following bound for $u=\mathcal{T}(\tilde u)$,
\begin{equation}
\label{notFar}
\sup_{0\leq t\leq T^\ve}\|u(t)\|_{H^2(I)}^2
\leq  (e^{T^\ve}-1)C_1
+e^{T^\ve} \|u_0\|_{H^2(I)}^2,
\end{equation}
with $C_1$ independent of $\tilde u\in L^\infty(0,T^\ve;K)$. Choosing $T^\ve$ sufficiently small we 
get that $\|u(t)\|_{H^2(I)}\leq M$ for $0<t<T^\ve$. 

Finally, multiply
\eqref{linear} by $(u-u_0)$ and integrate. After integrating by parts and using the fact that $\|u\|_{H^2(I)}\leq M$ we obtain
$$ 
\sup_{0\leq t\leq T^\ve}\|u(t)-u_0\|_{L^2(I)}^2\leq C_3(e^{T^\ve}-1).
$$
Then using the interpolation inequality 

$$
\|u-u_0\|^2_{C(I)}\leq C\|u-u_0\|_{H^2(I)}\|u-u_0\|_{L^2(I)}
$$ 
we get the bound
\begin{equation}
\label{Linftyblizko}
\sup_{0\leq t\leq T^\ve}\|u(t)-u_0\|_{C(I)}^4\leq C_4\|u(t)-u_0\|_{L^2(I)}^2\leq C_5(e^{T^\ve}-1).
\end{equation}

Now by \eqref{contend} and \eqref{Linftyblizko} we see that, possibly after passing to a smaller $T^\ve$, 
$\mathcal{T}$ maps $K$ into $K$ and it is a contraction on $K$.
\end{proof}

\noindent{\bf Step 3.} {\em Regularized equation: a priori estimates,  existence on time interval independent of $\ve$, and limit as $\ve\to 0$}.\\
In this step we derive a priori estimates which imply existence of a solution of \eqref{eqnfordve} on a time interval 
independent of $\ve$. 
These estimates are also used to pass to the $\ve\to 0$ limit. 


 \begin{lemma}\label{aprioriestH2uniform}
 Assume that $u_0\in H^2_{per}$ and $\| u_0 \|_{L^\infty(I)}<\delta_0$. Let $u^\ve$ solve 
 \eqref{eqnfordve} on a time interval $[0,T^\ve]$ with initial data $u^\ve(0)=u_0$, 
 and let $u^\ve$ satisfy $|u^\ve(\sigma,t)|\leq \delta_0$ on $I\times T^\ve$.
 Then 
 \begin{equation}
\label{unifH2bound}
\|u^\ve_{\sigma\sigma}\|_{L^2(I)}^2\leq a(t),
\end{equation} 
where $a(t)$ is the solution of 
\begin{equation}
\label{supersol1a} 
\dot a=2P a^3+2Q, \quad a(0)=\|(u_0)_{\sigma\sigma}\|_{L^2}^2
\end{equation}
(continued by $+\infty$ after the 
blow up time),
and  $0<P<\infty$, $0<Q<\infty$ are  independent of $\ve$ and $u_0$.
\end{lemma}


\begin{proof} For brevity we adopt the notation $u:=u^\ve,\; V:=V^\ve,\; S:=S^\ve$. 
Differentiate equation \eqref{eqnfordve} in $\sigma$ to find that
\begin{equation}\label{apriori1}
\begin{aligned}
u_{\sigma t}  + \ve \frac{\partial^5 u}{\partial \sigma^5} &-\frac{S}{1-u \kappa_0} \Phi^\prime (V) V_\sigma = 
\frac{S}{1-u \kappa_0} (\kappa(u))_\sigma\\
&+
\Bigl(\frac{S}{1-u \kappa_0}\Bigr)_\sigma \Bigl(\kappa(u)-\Phi(V)-\frac{2\pi}{L[u]}-
\frac{1}{L[u]} \int_I \Phi(V)S d\sigma\Bigr).
\end{aligned}  
\end{equation}
Next we rewrite \eqref{eqnfordve} in the form $V = \Phi(V)+\kappa - \lambda$ to calculate $V_\sigma$,  
$$(1- \Phi'(V))V_\sigma = \kappa_\sigma$$
whence
\[
V_\sigma = \frac{1}{1- \Phi^\prime(V)} \kappa_\sigma.
\]
Now we substitute this in \eqref{apriori1} to find that
 \begin{equation}\label{aprioriFinal}
\begin{aligned}
u_{\sigma t}  + \ve \frac{\partial^5 u}{\partial \sigma^5} &=
\frac{u_{\sigma\sigma\sigma}}{S^2(1-\Phi^\prime(V))}
-\frac{\Phi^\prime(V)+2}{(1-\Phi^\prime(V))S^4} u_\sigma u_{\sigma\sigma}^2 +
A(\sigma,V, u,u_\sigma)u_{\sigma\sigma}
\\
&+B(\sigma,V, u,u_\sigma)+
\Bigl(\frac{S}{1-u \kappa_0}\Bigr)_\sigma \Bigl(\Phi(V)-\frac{2\pi}{L[u]}-
\frac{1}{L[u]} \int_I \Phi(V)S d\sigma\Bigr),
\end{aligned}  
\end{equation}
where $A(\sigma,V, u,p)$ and $B(\sigma,V,u,p)$ are bounded continuous 
functions on $I\times \mathbb{R}\times [-\delta_0,\delta_0] \times \mathbb{R}$. 

Multiply \eqref{aprioriFinal} by $u_{\sigma\sigma\sigma}$ and integrate over  $I$, integrating by parts
on the left hand side. We find after rearranging terms and 
setting $\gamma:=\sup (1-\Phi^\prime(V))$ ($0<\gamma<\infty$),
\begin{equation}
\label{TrudnayaOtsenka}
\begin{aligned}
 \frac{1}{2}\frac{d}{dt}\|u_{\sigma\sigma}\|_{L^2(I)}^2 + 
 \ve \|u_{\sigma\sigma\sigma\sigma}\|_{L^2(I)}^2&+
 \frac{1}{\gamma}\int_I |u_{\sigma\sigma\sigma}|^2\frac{d\sigma}{S^2}\leq C 
 \int_I |u_{\sigma\sigma}|^2|u_{\sigma\sigma\sigma}|\frac{d\sigma}{S^3}\\
 &+ C_1\bigl(\|u_{\sigma\sigma}\|_{L^2(I)}^3+1\bigr)\Bigl(\int_I |u_{\sigma\sigma\sigma}|^2\frac{d\sigma}{S^2}\Bigr)^{1/2},
\end{aligned}
\end{equation}
where we have also used the inequality $\|u_\sigma\|_{L^\infty(I)}^2\leq |\,I\,| \|u_{\sigma\sigma}\|^2_{L^2(I)}$ and estimated various terms with the help of the Cauchy-Schwarz inequality. Next we estimate the first term in the right hand side of 
\eqref{TrudnayaOtsenka} as follows
\begin{equation}
\label{PromezhOtsenka}
\int_I |u_{\sigma\sigma}|^2|u_{\sigma\sigma\sigma}|\frac{d\sigma}{S^3}\leq 
C_\gamma \int_I |u_{\sigma\sigma}|^4\frac{d\sigma}{S^4}
+\frac{1}{4C_\gamma}  \int_I |u_{\sigma\sigma\sigma}|^2\frac{d\sigma}{S^2}
\end{equation}
and apply the following interpolation type inequality, 
whose proof is given in Lemma \ref{interplemma}: for all 
$u\in H^3_{per}(I)$ such that $|u|\leq \delta_0$ it holds that
\begin{equation}
\label{firstInterpolation}
\int_I |u_{\sigma\sigma}|^4\frac{d\sigma}{S^4} \leq \mu \int_I |u_{\sigma\sigma\sigma}|^2 \frac {d\sigma}{S^2}+C_2\Bigl(\int u_{\sigma\sigma}^2d\sigma\Bigr)^3+C_3\quad \forall \mu>0,
\end{equation}
where $C_2$ and $C_3$ depend only on $\mu$. Now we use  \eqref{PromezhOtsenka}  and 
\eqref{firstInterpolation} in \eqref{TrudnayaOtsenka}, and estimate the last term of 
\eqref{TrudnayaOtsenka} with the help of Young's inequality to derive that
\begin{equation}
\label{FinalnayaOtsenka}
 \frac{1}{2}\frac{d}{dt}\|u_{\sigma\sigma}\|_{L^2(I)}^2 + 
 \ve \|u_{\sigma\sigma\sigma\sigma}\|_{L^2(I)}^2+
 \frac{1}{4\gamma}\int_I |u_{\sigma\sigma\sigma}|^2\frac{d\sigma}{S^2}\leq P\|u_{\sigma\sigma}\|^6_{L^2(I)}
 +Q.
\end{equation}
Then by a standard comparison argument for ODEs $\|u_{\sigma\sigma}\|_{L^2(I)}^2\leq a(t)$, where $a$ solves \eqref{supersol1a}. The Lemma is proved.
\end{proof}
\begin{lemma}\label{interplemma} Assume that $u\in H^3_{per}(I)$ and $|u|\leq \delta_0$ on $I$. Then
\eqref{firstInterpolation} holds with $C_2$ and $C_3$ depending on $\mu$ only.
\end{lemma}

\begin{proof}
The following straightforward bounds will be used throughout the proof,
\begin{equation}
\label{AUX1S}
\frac{1}{\sqrt{2}}(|u_\sigma|+(1-\delta_0\|\kappa_0\|_{L^\infty(I)}))\leq S(u) \leq |u_\sigma|+1,
\end{equation}
 \begin{equation}
 \label{AUX2S}
|(S(u))_\sigma|\leq \frac{1}{S} (|u_\sigma||u_{\sigma\sigma}|+C|u_\sigma|+C_1).
 \end{equation}

Note that 
\begin{equation}\label{Otsenka_lem21}
\int_I u_{\sigma\sigma}^4\frac{d\sigma}{S^4} \leq C \left\| u^2_{\sigma\sigma}/S\right\|_{L^\infty(I)} \int_I u_{\sigma\sigma}^2d\sigma,
\end{equation}
where $C>0$ is independent of $u$. 
Since $\int_I u_{\sigma\sigma} d\sigma=0$ we have 
\begin{equation}
\label{Otsenka_lem22}
\| u_{\sigma\sigma}^2/S\|_{L^\infty(I)} \leq  \int_I|(u_{\sigma\sigma}^2/S)_\sigma|d\sigma.
\end{equation}
Next we use \eqref{AUX1S}, \eqref{AUX2S} to estimate the right hand side of  \eqref{Otsenka_lem22} with the help of the 
Cauchy-Schwarz inequality,
\begin{equation}
\label{Otsenka_lem23}
\begin{aligned}
\int_I|(u_{\sigma\sigma}^2/S)_\sigma|d\sigma&\leq 
2 \int_I\left|u_{\sigma\sigma\sigma}u_{\sigma\sigma}\right|\frac{d\sigma}{S} + 
C\left(\int_I |u_{\sigma\sigma}|^3 \frac{d\sigma}{S^2}+\int_I u_{\sigma\sigma}^2d\sigma\right) \\
&\leq 2 \left(\int_Iu^2_{\sigma\sigma\sigma}\frac{d\sigma}{S^2}\right)^{1/2} \left(\int_I u_{\sigma\sigma}^2 d\sigma \right)^{1/2} \\
 &+ C\left(\int_Iu^4_{\sigma\sigma} \frac{d\sigma}{S^4}\right)^{1/2} \left(\int_I u_{\sigma\sigma}^2 d\sigma\right)^{1/2}+ 
C \int_I u_{\sigma\sigma}^2d\sigma.
\end{aligned}
\end{equation}
Plugging \eqref{Otsenka_lem22}-\eqref{Otsenka_lem23} into \eqref{Otsenka_lem21} and using Young's inequality yields
\begin{equation*}
\begin{aligned}
\int_I u_{\sigma\sigma}^4\frac{d\sigma}{S^4}
\leq& 2 \left(\int_Iu^2_{\sigma\sigma\sigma} \frac{d\sigma}{S^2}\right)^{1/2} \left(\int_I u_{\sigma\sigma}^2d\sigma \right)^{3/2} \\
&+ \frac{1}{2}\int_I u^4_{\sigma\sigma}\frac{d\sigma}{S^4} +  C\left(\int_I u_{\sigma\sigma}^2d\sigma \right)^{3}+C.
\end{aligned}
\end{equation*}
Finally, using here Young's inequality once more time we deduce \eqref{firstInterpolation}.
\end{proof}

Consider a time interval $[0,T^*]$ with $T^*>0$ slightly smaller than the blow up  time $T^{bu}$ of \eqref{supersol1a}.
As a byproduct of Lemma  \ref{aprioriestH2uniform} (cf. \eqref{FinalnayaOtsenka}) we then have
for any $0<T\leq \min\{T^\ve,T^*\}$
\begin{equation}
\label{byprod1}
\sup_{t\in[0,T]}\|u_{\sigma\sigma}\|^2_{L^2(I)}+ \ve \int_0^{T}\|u_{\sigma\sigma\sigma\sigma}\|_{L^2(I)}^2 dt+
 \int_0^{T}\int_I \|u_{\sigma\sigma\sigma}\|^2_{L^2(I)}dt\leq C
\end{equation}
where $u:=u^\ve(\sigma,t)$ is the solution of \eqref{eqnfordve} described in Proposition \ref{prop1}, 
and $C$ is independent of $\ve$ and $T$. 
In order to show that the solution of  \eqref{eqnfordve} exists on a time interval independent of $\ve$, it 
remains to obtain a uniform in $\ve$ estimate on the growth of  $\|u^\ve-u_0\|_{C(I)}$ in time. Arguing as in the 
end of the proof of  Proposition \ref{prop1} and using \eqref{byprod1} one can  prove

\begin{lemma}
\label{afteraprioriestinC}
Assume that $u_0\in H^2_{per}(I)$,  $\| u_0 \|_{L^\infty(I)}<\delta_0$ and assume that the solution 
$u^\ve$ of \eqref{eqnfordve} satisfies $|u^\ve(\sigma,t)|\leq \delta_0$ on $I\times [0,T^\ve]$.   
Then for all $0<t\leq \min\{T^\ve, T^*\}$,
\begin{equation}
\label{uniformw12bound}
\|u^\ve-u_0\|_{C(I)}^4\leq C(e^{t}-1)
\end{equation} 
where $C$ is independent of $\ve$. In particular,  there exists $0<T^{**}\leq T^*$, independent of $\ve$, 
such that $\sup \bigl\{\|u^\ve(t)\|_{L^\infty(I)};\,  0\leq t\leq \min\{T^\ve,T^{**}\}\bigr\}< \delta_0$.
\end{lemma} 

Combining Proposition \ref{prop1} with Lemma~\ref{aprioriestH2uniform} and Lemma~\ref{afteraprioriestinC} we 
see that the solution $u^\ve$ of \eqref{eqnfordve} exists on the time interval $[0,T^{**}]$ and  \eqref{byprod1}
holds with $T=T^{**}$. Now, it is not hard to pass to the limit $\ve\to 0$ in  \eqref{eqnfordve}.  
Indeed, exploiting \eqref{byprod1} we 
see that, up to extracting a subsequence, $u^\ve\rightharpoonup u$ weak star in $L^\infty(0,T^{**}; H^2_{per}(I))$ as $\ve\to 0$. Using
 \eqref{byprod1} in \eqref{evoMotion} we also conclude that  the family
 $\{u^\ve_t\}_{0<\ve\leq 1}$ is bounded in $L^2(I\times [0,T^{**}])$. Combining uniform estimates on $\|u^\ve_{\sigma\sigma\sigma}\|_{L^2(I\times [0,T])}$ (from \eqref{byprod1}) and $\|u_t^\ve\|_{L^2(I\times [0,T^{**}])}$ we deduce $u^\ve\to u$ strongly in $L^2(0,T^{**}; H^2_{per}(I))\cap C(0,T^{**};H^1_{per}(I))$. Moreover, $u^\ve_t+\ve u^\ve_{\sigma\sigma\sigma\sigma}\rightharpoonup u_t$ weak star in $L^\infty(0,T^{**}; L^2(I))$ and $u^\ve_t+\ve u^\ve_{\sigma\sigma\sigma\sigma}\to u_t$ strongly 
in $L^2(0,T^{**}; H^2_{per}(I))$. Thus, in the limit we obtain a solution $u$ of \eqref{deqn}. That is we have proved 

\begin{theorem}(Existence for smooth initial data)\label{thm_smooth_in_data}
For any $u_0\in H^2_{per}(I)$ such that $\|u_0\|_{L^\infty(I)}<\delta_0$ there exists a solution  $u$ of \eqref{deqn} on a time interval 
$[0,T]$, with $T>0$ depending on $u_0$.
\end{theorem} 

\begin{remark} \label{Remark_universal}
Note that the time interval  in Theorem \ref{thm_smooth_in_data} can by chosen 
universally for all $u_0$ such that $\|u_0\|_{L^\infty(I)}\leq \alpha<\delta_0$ and $\|u_0\|_{H^2(I)}\leq M<\infty$, that 
is $T=T(M,\alpha)>0$.
\end{remark}

\noindent{\bf Step 4.} {\em Passing to $W^{1,\infty}$ initial data} \\
In this step we consider a solution  $u$ of \eqref{deqn} granted by Theorem \ref{thm_smooth_in_data}
and show that the requirement on $H^2$ smoothness of the initial data can be weakened to $W^{1,\infty}$.
To this end 
we pass to  the limit in \eqref{deqn} with an approximating sequence of smooth initial data.

The following result establishes a bound on $\|u\|_{L^{\infty}(I)}$  independent of the $H^2$-norm  
of initial data (unlike in Lemma \ref{afteraprioriestinC}) and 
provides also an estimate for $\|u\|_{W^{1,\infty}(I)}$.
\begin{lemma} 
\label{Lemma_various_uniformbounds} 
Let $u$ be a solution of \eqref{deqn} (with initial value $u_0$) on the interval $[0,T]$ satisfying $\|u(t)\|_{L^\infty(I)}\leq\delta_0$ for all $t<T$. 
Then 
\begin{equation}
\label{Thefirstl_inf_bound}
\|u(t)\|_{L^\infty(I)}\leq \|u_0\|_{L^\infty(I)}+Rt
\end{equation}
where $R\geq 0 $ is a constant independent of $u_0$. 
Furthermore, the following inequality holds
\begin{equation}
\label{Thesecond_inf_bound}
\|u_\sigma\|_{L^\infty(I)}\leq a_1(t),
\end{equation}
where $a_1(t)$ is the solution of
\begin{equation}
\label{supersol2a} 
\frac{d a_1}{d t}=P_1 a_1^2+Q_1 a_1+R_1, \quad a(0)=\|(u_0)_{\sigma}\|_{L^\infty(I)}
\end{equation}
(continued  by $+\infty$ after the blow up time) and $P_1$, $Q_1$, $R_1$ are 
positive constants independent  of $u_0$.
\end{lemma}

\begin{proof} Both bounds \eqref{Thefirstl_inf_bound} and 
\eqref{Thefirstl_inf_bound} are established by using the maximum 
principle. 

Consider $g(\sigma,t)=g(t):=\|u_0\|_{L^\infty(I)}+Rt$, where $R>0$ is to be specified. Since $\Phi$ is bounded it holds that
$
g_t-\Phi(g_t)\geq R-\sup\Phi 
$. Assuming that $u-g$ attains its maximum on $(0,t]\times I$, say at $(\sigma_0, t_0)$, we have
$u_\sigma(\sigma_0,t_0)=0$, $u_{\sigma\sigma}(\sigma_0,t_0)\leq 0$, $g_t(\sigma_0,t_0)\leq u_t(\sigma_0,t_0)$ and $S/(1-u\kappa_0)=1$. Using \eqref{deqn} and
monotonicity of the function $V-\Phi(V)$ we also get for $t=t_0$, $\sigma=\sigma_0$
$$
\begin{aligned}
R-\sup\Phi \leq g_t-\Phi(g_t)& \leq u_t -
\frac{S}{1-u \kappa_0}\Phi\Bigl(\frac{1-u \kappa_0}{S}u_t\Bigr)
\\
&\leq\frac{\kappa_0}{1-u \kappa_0}-\frac{1}{L[u]}\Bigl( \int_I \Phi\left(\frac{1-u \kappa_0}{S} u_t\right)Sd\sigma+2\pi\Bigr).
\end{aligned}
$$ 
The last term in this inequality is uniformly bounded for all functions $u$ satisfying $|u|\leq \delta_0$, and 
therefore $u-g$ cannot attain its maximum  on  $(0,t]\times I$ when $R$ is sufficiently large, 
hence  
$u(\sigma,t)\leq \|u_0\|_{L^\infty(I)}+Rt$. Similarly one proves that $u(\sigma,t)\geq -\|u_0\|_{L^\infty(I)}-Rt$.

To prove \eqref{Thesecond_inf_bound} we write down the equation obtained 
by differentiating \eqref{deqn} in $\sigma$ (cf. \eqref{aprioriFinal}), 
\begin{equation}
\label{aprioriFinal1}
\begin{aligned}
u_{\sigma t}  -
\frac{u_{\sigma\sigma\sigma}}{S^2(1-\Phi^\prime(V))}&+
\frac{\Phi^\prime(V)+2}{(1-\Phi^\prime(V))S^4} u_\sigma u_{\sigma\sigma}^2 =
Au_{\sigma\sigma}
\\
&+B+
\Bigl(\frac{S}{1-u \kappa_0}\Bigr)_\sigma \Bigl(\Phi(V)-\frac{2\pi}{L[u]}-
\frac{1}{L[u]} \int_I \Phi(V)S d\sigma\Bigr),
\end{aligned}  
\end{equation}
where we recall that $A(\sigma,V, u,p)$ and $B(\sigma,V,u,p)$ are bounded continuous 
functions on $I\times \mathbb{R}\times [-\delta_0,\delta_0] \times \mathbb{R}$. Consider the function $u_\sigma(\sigma,t)-a_1(t)$, with $a_1(t)$ satisfying $a_1(0)=\|u_\sigma(0)\|_{L^{\infty}(I)}$.  
If this function attains its maximum over $I\times (0,t]$ at a point 
$(\sigma_0, t_0)$ with $t_0>0$, we have at this point
\begin{equation}
\label{ContradictionPointwise}
\begin{aligned}
\frac{ d a_1}{d t}\leq u_{\sigma t}&\leq B+
\Bigl(\frac{S}{1-u \kappa_0}\Bigr)_\sigma \Bigl(\Phi(V)-\frac{2\pi}{L[u]}-
\frac{1}{L[u]} \int_I \Phi(V)S d\sigma\Bigr)\\
&< C_1+C_2u_\sigma^2+C_3 u_\sigma,
\end{aligned}  
\end{equation}
where $C_1$, $C_2$, $C_3$ are some positive constants independent  of $u$. 
We see now that for $P_1\geq C_2$, $Q_1\geq C_3$ and $R_1\geq C_1$, 
either 
$u_\sigma(\sigma_0,t_0)\leq a_1(t_0)$ or  inequality
\eqref{ContradictionPointwise} contradicts  \eqref{supersol2a}. 
This yields 
that for all $\sigma\in I$ and all $t< \sup\{\tau; a_1(\tau)\ \text{is finite}\}$, 
$$
u_\sigma(\sigma,t)-a_1(t)\leq \max\{0,\max_{\sigma\in I} u_\sigma(\sigma,0)-a_1(0)\}= 0.
$$ 
The lower 
bound for  $u_\sigma(\sigma,t)$ is proved similarly.
\end{proof}

Lemma \ref{Lemma_various_uniformbounds} shows that for some $T_1>0$
the solution $u$ of \eqref{deqn} satisfies 
\begin{equation}
\label{Linfsummary}
\|u\|_{L^\infty(I)}<\delta_0 \ \text{and}\ \|u_\sigma\|_{L^\infty(I)}\leq M_1 ,
\end{equation}
when $0<t\leq \min\{T,T_1\}$, 
where $T$ is the maximal time of existence of $u$. Moreover,
$T_1$ and constant $M_1$ in \eqref{Linfsummary} depend only on 
the choice of constants $0<\alpha<1$ and $M>0$ in the bounds for the 
initial data $|u_0|\leq \alpha$ and $\|(u_0)_\sigma\|_{L^\infty(I)}\leq M$. 
We prove next 
that actually $T\geq T_1$. 

\begin{lemma}
\label{Holder_continuity}
Let $u$ solve \eqref{deqn} and let \eqref{Linfsummary} hold on $0\leq t\leq T_2$, where $0<T_2\leq T_1$. 
Then for any $\delta>0$ 
\begin{equation}
\label{HolderBound}
|u_\sigma(\sigma^\prime,t)-u_\sigma(\sigma^{\prime\prime},t)|\leq C_\delta |\sigma^\prime-\sigma^{\prime\prime}|^\vartheta
\ \text{when} \ \delta\leq t\leq T_2,
\end{equation}
where $0<\vartheta<1$ and $C_\delta$ depend only on $\delta$ and constant $M_1$ in  \eqref{Linfsummary}.
\end{lemma}

\begin{proof} Recall from Lemma \ref{lem1} that $\Psi$ denotes the inverse function of $V-\Phi(V)$, and that
\begin{equation*}
\lambda(t):=\frac{1}{L[u]}\Bigl(2\pi+\int_I \Phi(V)S d\sigma\Bigr),
\end{equation*}
so that the solution $V$ of the equation $V=\Phi(V)+\kappa-\lambda$ is given by $\Psi(\kappa-\lambda)$. This allows 
us
to write the equation \eqref{aprioriFinal1} for $v:=u_\sigma$ in the form of a quasilinear parabolic equation,
\begin{equation}
\label{convenientForm}
v_t=\bigl(a(v_\sigma,\sigma,t)\bigr)_\sigma-\frac{\lambda(t) u_\sigma}{S(u){(1-u \kappa_0)}} v_\sigma 
-\lambda(t)(u\kappa_0)_\sigma \frac{u_\sigma^2}{S(u)(1-u \kappa_0)^2},
\end{equation}
where 
$$
a(p,\sigma,t)= \frac{S(u)}{1-u\kappa_0}\Bigl(\Phi\bigl(\Psi[\bar\kappa(p,\sigma,t)-\lambda(t)]\bigr)+\bar\kappa(p,\sigma,t)\Bigr)
$$
with
$$
\bar\kappa(p,\sigma,t)=
\frac{1}{S^3(u)} \Bigl((1-u \kappa_0) p + 2\kappa_0 u_\sigma^2+(\kappa_0)_\sigma u_\sigma u  + 
\kappa_0(1-u \kappa_0)^2\Bigr).
$$
Note that setting $V:=\Psi[\bar\kappa(p,\sigma,t)-\lambda(t)]$ we have
$$
\frac{\partial a}{\partial p}(p,\sigma, t)=  \frac{S(u)}{1-u\kappa_0}\Bigl(\frac{\Phi^\prime(V)}{1-\Phi^\prime(V)}+1\Bigr)
\frac{\partial \bar\kappa}{\partial p}(p,\sigma,t)=\frac{1}{S^2(u)(1-\Phi^\prime(V))}.
$$
It follows that if  \eqref{Linfsummary} holds then the quasilinear divergence form parabolic 
equation \eqref{convenientForm} satisfies all necessary conditions to apply classical results on 
H\"older continuity of bounded solutions, see e.g. \cite{Lad68} [Chapter V, Theorem 1.1]; 
this completes  the proof.  
\end{proof}

The property of H\"older continuity established in Lemma \ref{Holder_continuity} allows us to prove the 
following important result

\begin{lemma}
\label{LemmaProGlobalGronwal}
Let $u$ solve \eqref{deqn} and assume that 
\eqref{Linfsummary} holds then for any $\tau\geq \delta>0$
\begin{equation}
\label{GronFinal}
\|u_{\sigma\sigma}(t)\|^2_{L^2(I)}\leq  (\|u_{\sigma\sigma}(\tau)\|^2_{L^2(I)}+1)e^{P_2 (t-\tau)} 
\quad   \text{when}\quad  t\geq \tau,
\end{equation}
where $P_2$ depends only on $\delta$ and constant $C$ in \eqref{Linfsummary} (the latter constant
depend not on the $H^2$ norm of $u_0$ but on its norm in the space $W^{1,\infty}(I)$). 
\end{lemma}

\begin{proof} Introduce smooth cutoff functions  $\phi_n(\sigma)$ satisfying 
\begin{equation}
\begin{cases}
\phi_n(\sigma)=1 \ \text{on} \ [1/n,2/n]\\
0\leq \phi_n(\sigma)\leq 1 \ \text{on} \ [0,1/n]\cup[2/n,3/n]\\
\phi_n(\sigma)=0 \ \text{otherwise},
\end{cases}
\quad \text{and}\quad
|(\phi_n)_\sigma|\leq Cn.
\end{equation}

Consider $t\geq \delta$. Multiply \eqref{aprioriFinal1} by $-\bigl( \phi_n^2 u_{\sigma\sigma}\bigr)_\sigma$ and integrate in
$\sigma$, integrating by parts in the first term. 
 We obtain
  \begin{equation}
  \label{Schitaem}
 \begin{aligned}
  \frac{1}{2}\frac{d}{dt}
  \int_I \phi_n^2u_{\sigma\sigma}^2 d\sigma+
  \frac{1}{\gamma^\prime}\int_I \phi_n^2u_{\sigma\sigma\sigma}^2 d\sigma
 & \leq C\int_I \phi_n^2(u_{\sigma\sigma}^2+|u_{\sigma\sigma}|)|u_{\sigma\sigma\sigma}| d\sigma\\
 &+
  Cn\int_{I} \phi_n\bigl(|u_{\sigma\sigma\sigma}||u_{\sigma\sigma}|+|u_{\sigma\sigma}|^2+|u_{\sigma\sigma}|^3\bigr) d\sigma,
 \end{aligned}
  \end{equation}
where $\gamma ^\prime >0$ and $C$ are independent of $u$ and $n$. Applying the Cauchy-Schwarz and Young's inequalities to 
various terms in the right hand side of \eqref{Schitaem}  leads to 
\begin{equation}
 \label{Schitaem1}
 \frac{1}{2}\frac{d}{dt}
 \int_I \phi_n^2u_{\sigma\sigma}^2 d\sigma+
 \frac{1}{2\gamma^\prime}\int_I \phi_n^2u_{\sigma\sigma\sigma}^2 d\sigma
\leq C_1 \int_I \phi_n^2u_{\sigma\sigma}^4 d\sigma
+C_2 n^2 \int_{{\rm s}(\phi_n)}(|u_{\sigma\sigma}|^2+1) d\sigma ,
 \end{equation}
where $C_1$, $C_2$ are independent of $u$ and $n$ and ${\rm s}(\phi_n)$ denotes the support of $\phi_n$.
Next we apply  the following 
interpolation type inequality (see, e.g., \cite{Lad68}, Chapter II, Lemma 5.4)
to the first term in the right hand side of \eqref{Schitaem1}:
\begin{equation}
\label{interpolLadyzh}
\begin{aligned}
 \int_I \phi_n^2u_{\sigma\sigma}^4 d\sigma\leq 
 C_3\bigl(&\sup\{|u_\sigma(\sigma^\prime)-u_\sigma(\sigma^{\prime\prime})|;\, \sigma^\prime,\sigma^{\prime\prime}\in 
 {\rm s}(\phi_n)\}\bigr)^2\\
 &
\times \Bigl(\int_I \phi_n^2u_{\sigma\sigma\sigma}^2 d\sigma+\int_{{\rm s}(\phi_n)}|u_{\sigma\sigma}|^2|(\phi_n)_\sigma|^2 d\sigma
 \Bigr).
\end{aligned}
\end{equation}
Now we use \eqref{HolderBound} in \eqref{interpolLadyzh} to bound 
$\sup\{|u_\sigma(\sigma^\prime)-u_\sigma(\sigma^{\prime\prime})|;\, \sigma^\prime,\sigma^{\prime\prime}\in 
 {\rm s}(\phi_n)\}$ 
 by $C_\delta(1/n)^\vartheta$ and choose $n$ so large  that
 $C_2C_3 C_\delta^2(1/n)^{2\vartheta}\leq 1/4\gamma^\prime$, then  \eqref{Schitaem1} becomes 
 \begin{equation}
 \label{Schitaem2}
  \frac{1}{2}\frac{d}{dt}
  \int_I \phi_n^2u_{\sigma\sigma}^2 d\sigma+
 \frac{1}{4\gamma^\prime}\int_I \phi_n^2u_{\sigma\sigma\sigma}^2 d\sigma
\leq C_4  n^2 \int_{{\rm s}(\phi_n)}(|u_{\sigma\sigma}|^2+1) d\sigma.
 \end{equation}
It is clear that we can 
replace  $\phi_n(\sigma)$ in \eqref{Schitaem2} by its translations $\phi_n(\sigma+k/n)$, $k\in \mathbb{Z}$,
then taking the sum of obtained inequalities we derive
 \begin{equation}
 \label{Schitaem3}
  \frac{1}{2}\frac{d}{dt}
  \int_I \overline{\phi}_n u_{\sigma\sigma}^2 d\sigma+
 \frac{1}{4\gamma^\prime}\int_I \overline \phi_n u_{\sigma\sigma\sigma}^2 d\sigma
\leq C_5  \Bigl(\int_{I}|u_{\sigma\sigma}|^2 d\sigma +1\Bigr),
 \end{equation}
where $C_5$ is independent of $u$, 
and 
$ \overline{\phi}_n= \sum_k \phi_n^2(\sigma+k/n)$. Note that $1\leq  \overline{\phi}_n\leq 3$, therefore
applying Gr\"onwall's inequality to  \eqref{Schitaem3} we obtain \eqref{GronFinal}.
\end{proof}

\begin{corollary} 
\label{CorBoundFor3der}
Assume that  a solution $u$ of  \eqref{deqn} exists on $[0,T]$ for some $T>0$, and  
\eqref{Linfsummary} holds. Then, given an arbitrary positive $\delta<T$,  
we have 
\begin{equation}
\label{u_sss_norm}
\int_\tau^T\int_I |u_{\sigma\sigma\sigma}|^2d\sigma d t\leq \tilde{C}_\delta \|u_{\sigma\sigma}(\tau)\|^2_{L^2(I)},
\end{equation}
for all $\delta\leq\tau\leq T$, where $\tilde{C}_\delta$ depends only on $\delta$.
\end{corollary}
\begin{proof} The bound follows by integrating \eqref{Schitaem3} in time and using \eqref{GronFinal}. \end{proof}

Using Lemma \ref{LemmaProGlobalGronwal} and Lemma \ref{Lemma_various_uniformbounds}, taking into account also Remark \ref{Remark_universal},
we see  that solutions in Theorem \ref{thm_smooth_in_data} exist on a common interval $T=T(\alpha, M)$, 
provided  that $\|u_0\|\leq \alpha<1$ and $\|(u_0)_\sigma\|_{L^\infty(I)}\leq M$. 

In the following Lemma we establish an integral bound for $\|u_{\sigma\sigma}\|_{L^2(I)}$.

\begin{lemma} 
\label{LemmaIntegralBound}
Let $u$ be a solution \eqref{deqn} on $[0,T]$ satisfying \eqref{Linfsummary}. Then 
\begin{equation}
\label{IntegralBound}
\|u_{\sigma\sigma}\|^2_{L^2(I\times [0,T])} \leq C,
\end{equation}
where $C$ depends only on the constants in \eqref{Linfsummary}. 
\end{lemma}   

\begin{proof} To obtain \eqref{IntegralBound} one multiplies \eqref{deqn} by $u_{\sigma\sigma}$  integrates in $\sigma$, integrating by parts in 
the first term. Then  applying the Cauchy-Schwarz and Young's inequalities and  integrating in $t$ one derives \eqref{IntegralBound}, details are left 
to the reader.        
\end{proof}

Various estimates obtained in Lemmas \ref{Lemma_various_uniformbounds}, \ref{LemmaProGlobalGronwal}, and Corollary \ref{CorBoundFor3der} as well as Lemma \ref{LemmaIntegralBound} make it possible to pass to general initial data  $u_0 \in W^{1,\infty}_{per}(I)$.
Indeed, assume  $\alpha:=\|u_0\|_{L^\infty(I)}<1$ and let $M:=\|(u_0)_\sigma \|_{L^\infty(I)}$. Construct a sequence 
$u_0^k \rightharpoonup u_0$ converging weak star in $W^{1,\infty}_{per}(I)$
as $k\to\infty$, where $u_0^k\in H^2_{per}(I)$. This can be done in a standard way by taking convolutions with mollifiers
so that $\|u^k_0\|_{L^\infty(I)}\leq \alpha$ and $\|(u_0^k)_\sigma\|_{L^\infty(I)}\leq M$. Let $u^k(\sigma,t)$ be  solutions of 
\eqref{deqn} corresponding to the initial data $u^k_0(\sigma)$. We know that all these solutions exist on a common time interval $[0,T]$ and that we can choose $T>0$ such that  \eqref{Linfsummary} holds with a constant $C$ independent of $k$. By Lemma \ref{LemmaIntegralBound} the sequence $u^k(\sigma, t)$ is bounded in $L^2(0,T; H^2(I))$. Therefore, up to a subsequence, 
$u^k(\sigma,t)$ weakly converges to some function $u(\sigma,t)\in L^2(0,T; H^2(I))$. Using \eqref{deqn} we conclude 
that $u^k_t(\sigma,t)$ converge to $u_t(\sigma,t)$ weakly in $L^2(0,T; L^2(I))$. It follows, in particular, 
that  $u(\sigma,0)=u_0(\sigma)$. Next, let $\delta>0$ be sufficiently small. It follows from Lemma 
\ref{LemmaIntegralBound} that there exists $\tau\in [\delta,2\delta]$ such that $\|u_{\sigma\sigma}^k(\tau)\|\leq C/\delta$. 
Then by Lemma \ref{LemmaProGlobalGronwal} and  Corollary \ref{CorBoundFor3der} norms  
of $u^k$ in $L^\infty(2\delta,T; H^2_{per}(I))$ and $L^2(2\delta,T; H^3_{per}(I))$ are uniformly bounded. 
Thus $u^k(\sigma, t)$ converge to $u(\sigma,t)$ strongly in $L^2(2\delta,T; H^2_{per}(I))$. This in turn implies
that  $u^k_t(\sigma,t)$ converge to $u_t(\sigma,t)$ strongly in $L^2(0,T; L^2(I))$. Therefore the 
function  $u(\sigma,t)$ solves \eqref{deqn} on $[2\delta,T]$. Since $\delta>0$ can be chosen arbitrarily small, 
Theorem \ref{thm1} is completely proved.



\section{Non-existence of traveling wave solutions}
\label{sec:nonexistence}
The following result proves that smooth ($H^2$) non-trivial traveling wave solutions of \eqref{interface0} do not exist. The idea of the proof is to write equations of motion of the front and back parts of a curve, which is supposed to be a traveling wave solution, using a Cartesian parametrization. Next we show that it is not possible to form a closed, $H^2$-smooth curve from these two parts. We note that every traveling wave curve is always $H^2$ smooth since it is the same profile for all times up to translations.

\begin{theorem}\label{thm2}
Let $\Phi$ satisfy conditions of Theorem \ref{thm1}. If $\Gamma(\sigma,t)
$ is a family of closed curves which are a traveling wave solution of \eqref{interface0}, 
that is $\Gamma(\sigma,t)=\Gamma(\sigma,0)+vt$,
then $v=0$ and $\Gamma(\sigma,0)$ is a circle. \end{theorem}

\begin{proof}
It is clear that if $v=0$ then a circle is the unique solution of \eqref{interface0}.
 Let $\Gamma(\sigma,t)$ be a traveling wave solution of \eqref{interface0} with non-zero velocity $v$. By Theorem \ref{thm1}, $\Gamma(\cdot,t)$ is smooth ($H^2$) for all $t>0$.  By rotation and translation, we may assume without loss of generality that $v_x=0$, $v_y=c$  with $c>0$ and that 
 $\Gamma(\sigma,t)$ is contained in the upper half plane for all $t\geq 0$. Let $\Gamma(\sigma_0,0)$ be a point of $\Gamma(\sigma,0)$ which is closest to the $x$-axis. Without loss of generality we assume that $\Gamma(\sigma_0,0)=0$.  Locally, we can represent $\Gamma(\sigma,t)$ as a graph over the $x$-axis, $y=y(x) + ct$. Observe that the normal velocity is given by
\begin{equation}
V
=\frac{c}{\sqrt{1+(y'(x)^2)}}
\end{equation}
and the curvature $\kappa$ is expressed as 
\begin{equation}
\kappa(x)= \frac{y''(x)}{(1+(y'(x))^2)^{3/2}}.
\end{equation}
Adopting the notation
\begin{equation}
f_\lambda^c(z) := \left(\frac{c}{\sqrt{1+z^2}}-\Phi\left(\frac{c}{\sqrt{1+z^2}}\right)+\lambda\right)(1+z^2)^{3/2},
\end{equation}
it follows that $y$ solves the equation
\begin{equation}\label{phieqn}
y'' = f_\lambda^c(y')
\end{equation}
where, by construction,
\begin{equation}
y(0)=y'(0)=0.
\end{equation}

Observe that \eqref{phieqn} is a second order equation for $y$ which depends only on $y'$. Thus, we may equivalently study $w:=y'$ which solves
\begin{equation}\label{weqn}
w' = f_\lambda^c(w),\;\;\;\; w(0)=0.
\end{equation}
Note that \eqref{weqn} is uniquely solvable on its interval of existence by Lipschitz continuity of $f_\lambda^c$. Further, the definition of $f_\lambda^c$ guarantees that $w$ has reflectional symmetry over the $y$-axis.

 If \eqref{weqn} has a global solution, $w$, then it is immediate that $y(x) := \int_0^x w(s)ds$ cannot describe part of a closed curve. As such we restrict to the case where $w$ has finite blow-up.
 
Assume the solution $w_B$ of \eqref{weqn} has finite blow-up, 
$w_B(x)\to +\infty$ as $x\to x^*_B$ for some $0<~x^*_B<~\infty$. Then $y_B(x):=\int_0^x w_B(s)ds$ has a vertical tangent vector at $\pm x^*_B$. To extend the solution beyond the point $x^*_B$, we consider $w_F$, the solution of \eqref{weqn} with right hand side $f^{-c}_{\lambda}$. As above we assume that $w_F$ has a finite blow-up at $x^*_F>0$. Defining $y_F(x):=\int_0^x w_F(s)ds$, we have the following natural transformation,
\begin{equation}
\hat{y}_F(x) :=
-y_F(x-(x^*_B-x^*_F))+y_B(x^*_B)+y_B(x^*_F).
\end{equation}
Note that 
gluing $\hat{y}_F$ to $y_B$ forms an $H^2$ smooth curve at the point $(x^*_B,y_B(x^*_B))$ if and 
only if $w_F(x)\to +\infty$ as $x\to x^*_F$. We claim that this is the unique, smooth extension of $y_B$ at $x^*_B$. To that end, consider the rotated coordinate system $(x,y)\mapsto (y,-x)$. 
In this frame, the traveling wave moves with velocity $v_x=c$, $v_y=0$, and can be locally represented as the graph
$x=x(y)+ct$, 
 $x(y)$ solving
\begin{equation}\label{rotatedeqn}
x'' = g_\lambda^c(x')
\end{equation}
with
\begin{equation}
g_\lambda^c(z):= \left(\frac{-cz}{\sqrt{1+z^2}}-\Phi\left(\frac{-cz}{\sqrt{1+z^2}}\right)+\lambda\right)(1+z^2)^{3/2}.
\end{equation}
As before, $g_\lambda^c$ is Lipschitz and so solutions of \eqref{rotatedeqn} are unique, establishing the claim.

To complete the proof, we prove that $x^*_F>x^*_B$, which guarantees that the graphs of $y_B(x)$ and $\hat{y}_F(x)$ can not smoothly meet at $-x^*_B$.  Due to the monotonicity of $V-\Phi(V)$ we have that for any $w$,
$
f^{c}_\lambda(w)>f^{-c}_\lambda(w)
$.
Thus $w'_B> w'_F$ for any fixed $w$. Since $w_B(0)=w_F(0)$, we deduce that $w_B(x)> w_F(x)$ for all $x > 0$. It follows that $x^*_F\geq x^*_B$. Let $x_2\in(0, x^*_B)$ and observe by continuity of $w_B$ that there exists $x_1\in (0,x_2)$ such that $w_B(x_1)=w_F(x_2)$. Consider the solution $\tilde{w}$ of
\begin{equation}
\tilde{w}' = f^c_\lambda(\tilde{w})\;\;\; \tilde{w}(x_2)=w_F(x_2).
\end{equation}
 Note that $\tilde{w}(x) = w_B(x-(x_2-x_1))$ and so the blow-up of $\tilde{w}(x)$ occurs at $x_B^*+x_2-x_1$. Since $\tilde{w}(x)\geq w_F(x)$ for all $x\in (x_2,x_B^*+x_2-x_1)$ it follows that $x^*_F\geq x_B^*+x_2-x_1>x_B^*$. 
 \end{proof}
 
 \begin{figure}[H]
 \centering
 \includegraphics[width=.5\textwidth]{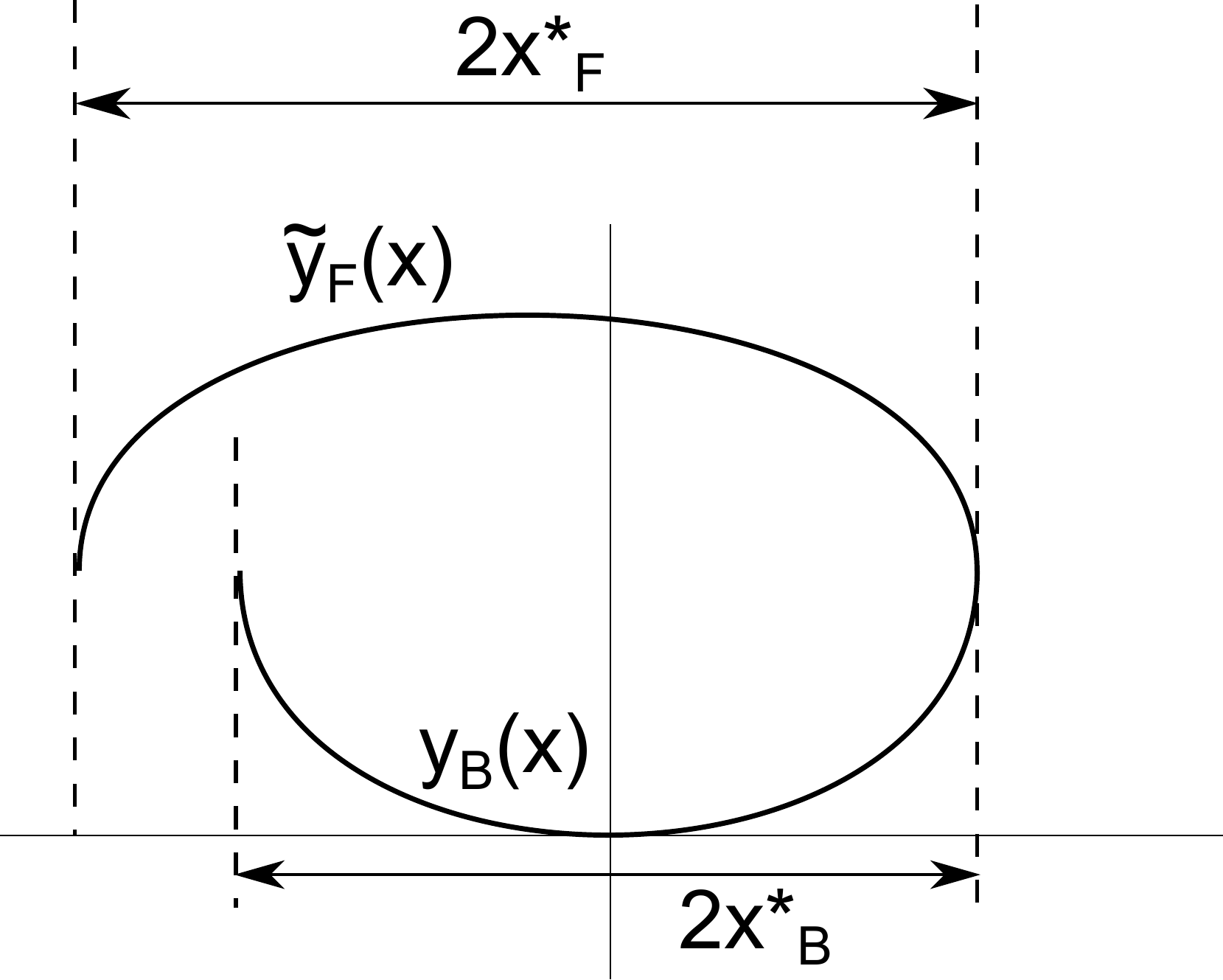}
 \caption{A closed curve cannot be a traveling wave solution}	
 \end{figure}

\section{Numerical simulation; comparison with volume preserving curvature motion}
\label{sec:numerics}

The preceding Section \ref{sec:existence} proves short time existence of curves propagating via \eqref{interface0}, and Section \ref{sec:nonexistence} shows the nonexistence of traveling wave solution. In this section we will numerically solve \eqref{interface0} and for that purpose we will introduce a new splitting algorithm. Using this algorithm, we will be able to study the long time behavior of the cell motion by numerical experiments, in particular, we find that both non-linearity and asymmetry of the initial shape will result in a net motion of the cell.



Specifically we numerically solve the equation \eqref{interface0} written so that the dependence on $\beta$ is explicit ($\beta \Phi_0(V)=\Phi(V)$, see remark \ref{rem:betaequiv}):
\begin{equation}\label{linearinterface}
V(s,t) = \kappa(s,t) + \beta \Phi_0(V(s,t)) - \frac{1}{|\Gamma(t)|}\int (\kappa(s',t)+\beta\Phi_0(V(s',t)))ds',
\end{equation}
We propose an algorithm and use it to compare curves evolving by \eqref{linearinterface} with $0<\beta<\beta_{cr}$, with curves evolving by volume preserving curvature motion ($\beta=0$). 



For simplicity, we assume that $\Phi_0(V)$ defined via \eqref{phi} has a Gaussian shape 
$
\Phi_1(V) := e^{-|V|^2}
$
(see Figure \ref{phicomparison})
which agrees well with the actual graph of $\Phi_0$. This significantly decreases computational time in numerical simulations.

\begin{figure}[h!]
\centering
\includegraphics[width=.5\textwidth]{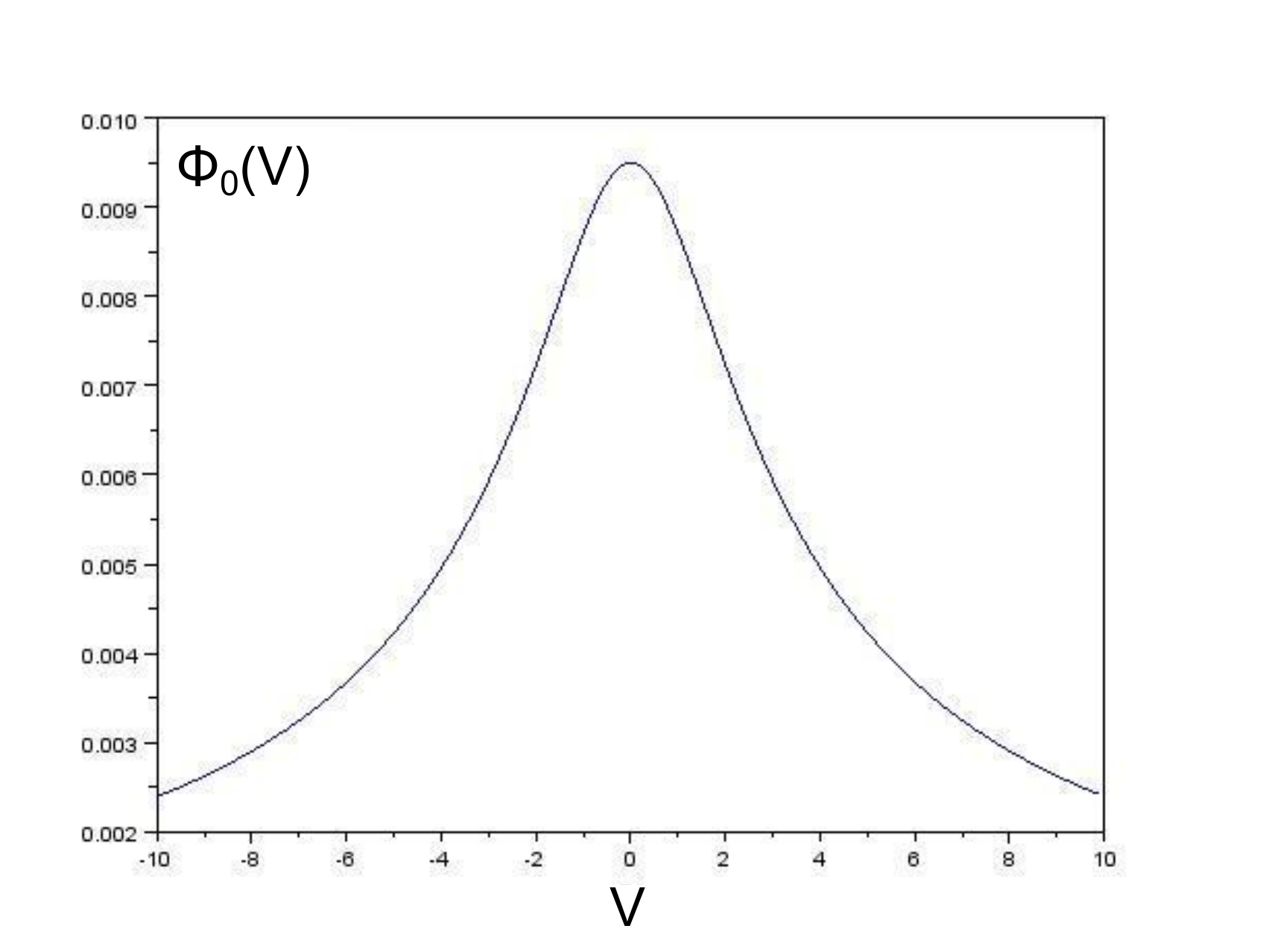}
\caption{Graph of $\Phi_0$}\label{phicomparison}
\end{figure}

\subsection{Algorithm to solve (\ref{linearinterface})}

In the case $\beta=0$ (corresponding to volume preserving curvature motion), efficient techniques such as level-set methods \cite{Osh88,Sme03} and diffusion generated motion methods \cite{Mer93,Ruu03} can be used to accurately simulate the evolution of curves by \eqref{linearinterface}.  There is no straightforward way to implement these methods when $\beta> 0$ since $V$ enters equation \eqref{linearinterface} implicitly.

%

Moreover, due to the non-local term, a naive discrete approximation of \eqref{linearinterface} leads to a system of non-linear equations. 
Instead of using a non-linear root solver, we introduce a {\em splitting scheme} which resolves the two main computational difficulties (non-linearity and volume preservation) separately.

In particular, we decouple the system by solving the $N$ local equations
\begin{equation}\label{localversion}
V_i = \kappa_i +\beta\Phi_1(V_i)-C,
\end{equation}
where $C$ is a constant representing the volume preservation constraint which must be determined. For $\beta<\beta_{cr}$, \eqref{localversion} can be solved using an iterative method which (experimentally) converges quickly. The volume constraint can be enforced by properly changing the value of $C$.

We recall the following standard notations. Let $p_i = (x_i,y_i)$, $i=1,\dots, N$ be a discretization of a curve. Then $h:=1/N$ is the grid spacing and
\begin{equation}
Dp_i := \frac{p_{i+1}-p_{i-1}}{2h}\text{ and  } D^2p_i := \frac{p_{i+1}-2p_i+p_{i-1}}{h^2} \;\;\; (h=1/N)
\end{equation}
are the second-order centered approximations of the first and second derivatives, respectively. Additionally, $(a,b)^\perp = (-b,a)$.

 We introduce the following algorithm for a numerical solution of \eqref{linearinterface}. \\

\begin{algorithm}
To solve \eqref{linearinterface} up to time $T>0$ given the initial shape $\Gamma(0)$. 
\begin{enumerate}
\item[{\em Step 1:}] (Initialization) 
Given a closed curve $\Gamma(0)$, discretize it by $N$ points $p^0_i=(x^0_i,y^0_i)$.

Use the shoelace formula to calculate the area of $\Gamma(0)$:
\begin{equation}\label{eqn:shoelace}
A^o = \frac{1}{2} \left| \sum_{i=1}^{n-1} x^0_i y^0_{i+1} + x^0_n y^0_1 - \sum_{i=1}^{n-1} x^0_{i+1}y^0_i - x^0_1 y^0_n \right|.
\end{equation}

Set time $t:=0$, time step $\Delta t>0$, and the auxiliary parameter $C:=0$. 

\item[{\em Step 2:}] (time evolution) If $t<T$, calculate the curvature at each point, $\kappa_i$ using the formula
\begin{equation}
\kappa_i = \frac{\operatorname{det}(Dp^t_i,D^2 p^t_i)}{\|Dp^t_i\|^3},
\end{equation}
where $\|\cdot\|$ is the standard Euclidean norm.
Use an iteration method to solve 
\begin{equation}\label{discreteV}
V^t_i = \kappa_i+\beta \Phi_1(V^t_i)-C
\end{equation}
to within a fixed tolerance $\varepsilon>0$.

Define the temporary curve 
\begin{displaymath}
p^{temp}_i := p^t_i+V^t_i \nu_i\Delta t,
\end{displaymath} 
where $\nu_i=(D p_i^0)^\perp/\|Dp_i^0\|$ is the inward pointing normal vector.

Calculate the area of the temporary curve $A^{temp}$ using the shoelace formula \eqref{eqn:shoelace} and compute the discrepancy 
\begin{displaymath}
\Delta A:=(A^{temp}-A^{o})\cdot (A^o)^{-1},
\end{displaymath}
If $|\Delta A|$ is larger than a fixed tolerance $\varepsilon$, adjust $C \mapsto C +\Delta A$ and return to solve \eqref{discreteV} with updated $C$. Otherwise define $p_i^{t+\Delta t} := p_i^{temp}$ and 
\begin{displaymath}
\Gamma(t+\Delta t):=\{p_i^{t+\Delta t}\},
\end{displaymath}

Let $t=t+\Delta t$, if $t<T$ iterate Step 2; else, stop.

\end{enumerate}
\end{algorithm}
In practice, we additionally reparametrize the curve by arc length after a fixed number of time steps in order to prevent regions with high density of points $p_i$ which could lead to potential numerical instability and blow-up.
\begin{remark}\label{iterativealg}
In Step 2, for $\beta\leq 1$, the right hand side of \eqref{discreteV} is contractive and thus we may guarantee convergence of the iterative solver to the solution of \eqref{discreteV}.
\end{remark}

We implement the above algorithm in C++ and visualize the data using Scilab. We choose the time step $\Delta t$ and spatial discretization step $\displaystyle h = \frac{1}{N}$ so that
\begin{equation*}
\frac{\Delta t}{h^2} \leq \frac{1}{2},
\end{equation*}
in order to ensure convergence.
Further, we take an error tolerance 
\begin{equation*}
\varepsilon = .0001
\end{equation*}
 in Step 2 for both iteration in \eqref{discreteV} and iteration in $C$.


\subsection{Convergence of numerical algorithm}
To validate results, we check convergence of the numerical scheme in the following way: taking a fixed initial curve $\Gamma(0)$, we discretize it with varying numbers of points: $N=2^m$ for $m=5,\dots,8$, and fix a final time step sufficiently large so that $\Gamma(k\Delta t)$ reaches a steady state circle.  Since there is no analytic solution to \eqref{interface0}, there is no absolute measure of the error. Rather, we define the error between successive approximations $N$ and $2N$. \newline
To this end, we calculate the center location, $C_N=(C_N^1,C_N^2)$, of each steady state circle as the arithmetic mean of the data points. Define the error between circles as
\begin{equation*}
err_N = \|C_N - C_{2N}\|_{\ell^2}.
\end{equation*}
Then, the convergence rate can be expressed as
\begin{equation*}
\rho := \lim_{N\to \infty} \rho_N = \lim_{N\to \infty} \log_2 \frac{err_N}{err_{2N}}.
\end{equation*}
We record our results in Table \ref{tab:1}, as well as in Figure \ref{circleconvzoom}. Note that $\rho\approx 1$, namely our numerical method is 
first-order.

\begin{table}[h]
\centering
\begin{tabular}{| c | l | c|}
\hline
$N$ &  $err_N$ & $\rho_N$ \\ \hline
32 & .2381754 & .9141615 \\ \hline
64 & .1263883 & .9776830 \\ \hline
128 & .0641793 &  1.0019052 \\ \hline
256 & .0320473 & - \\ \hline
\end{tabular}
\vspace*{3mm}
\caption{Table for convergence of numerical simulations}\label{tab:1}
\end{table}

%
%

\begin{figure}[h!]
\centering
\includegraphics[width = .7\linewidth]{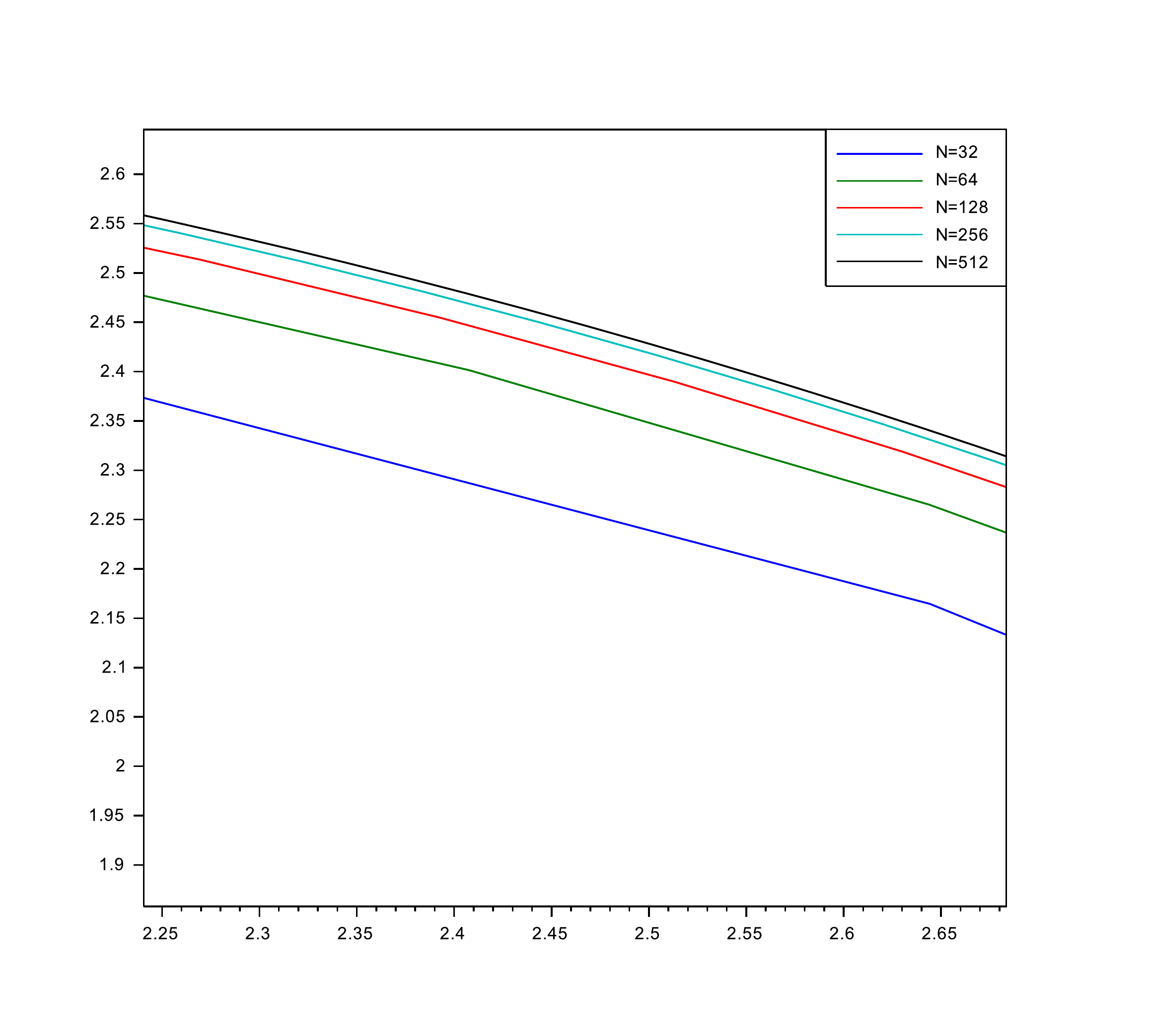}
\caption{Convergence of numerical simulation for decreasing mesh size $h=1/N$ (zoomed in)}\label{circleconvzoom}
\end{figure}

\subsection{Numerical experiments for the long time behavior of cell motion}

We next present two numerical observations for the subcritical case $\beta<\beta_{cr}$. \\
 
\noindent{\bf 1.} {\em If a curve globally exists in time, then it tends to a circle.}
	

That is, in the subcritical $\beta$ regime, curvature motion dominates non-linearity due to $\Phi(V)$. This is natural, since for small $\beta$ the equation \eqref{linearinterface} can be viewed as a perturbation of volume preserving curvature motion, and it has been proved (under certain hypotheses) that curves evolving via volume preserving curvature motion converge to circles \cite{Esc98}. \newline

 In contrast, the second observation distinguishes the evolution of \eqref{linearinterface} from volume preserving curvature motion.\\

\noindent{\bf 2. } {\em There exist curves whose centers of mass exhibit net motion on a finite time interval (transient motion).}\\

The key issue in cell motility is (persistent) net motion of the cell. Although Theorem \ref{thm2} implies that no non-trivial traveling wave solution of \eqref{interface0} exists, observation 2 implies that curves propagating via \eqref{linearinterface} may still experience a transient net motion compared to the evolution of curves propagating via volume preserving curvature motion. We investigate this transient motion quantitatively with respect to the non-linear parameter $\beta$ and the initial geometry of the curve.

\subsubsection{Quantitative investigation of observation 2}



Given an initial curve $\Gamma(0)$ discretized into $N$ points $p_i^0$ and given $0\leq \beta<\beta_{cr}$, we let $\Gamma_\beta(k\Delta t):=\{p_i^k\}$ be the curve at time $k\Delta t$, propagating by \eqref{linearinterface}. 
In particular, $\Gamma_0(k\Delta t)$ corresponds to the evolution of the curve by volume preserving curvature motion. 

Our prototypical initial curve $\Gamma(0)$ is parametrized by four ellipses and is sketched in Figure \ref{fig:initialcurves}.

\begin{figure}[h!]
\centering
\includegraphics[width=.4\linewidth]{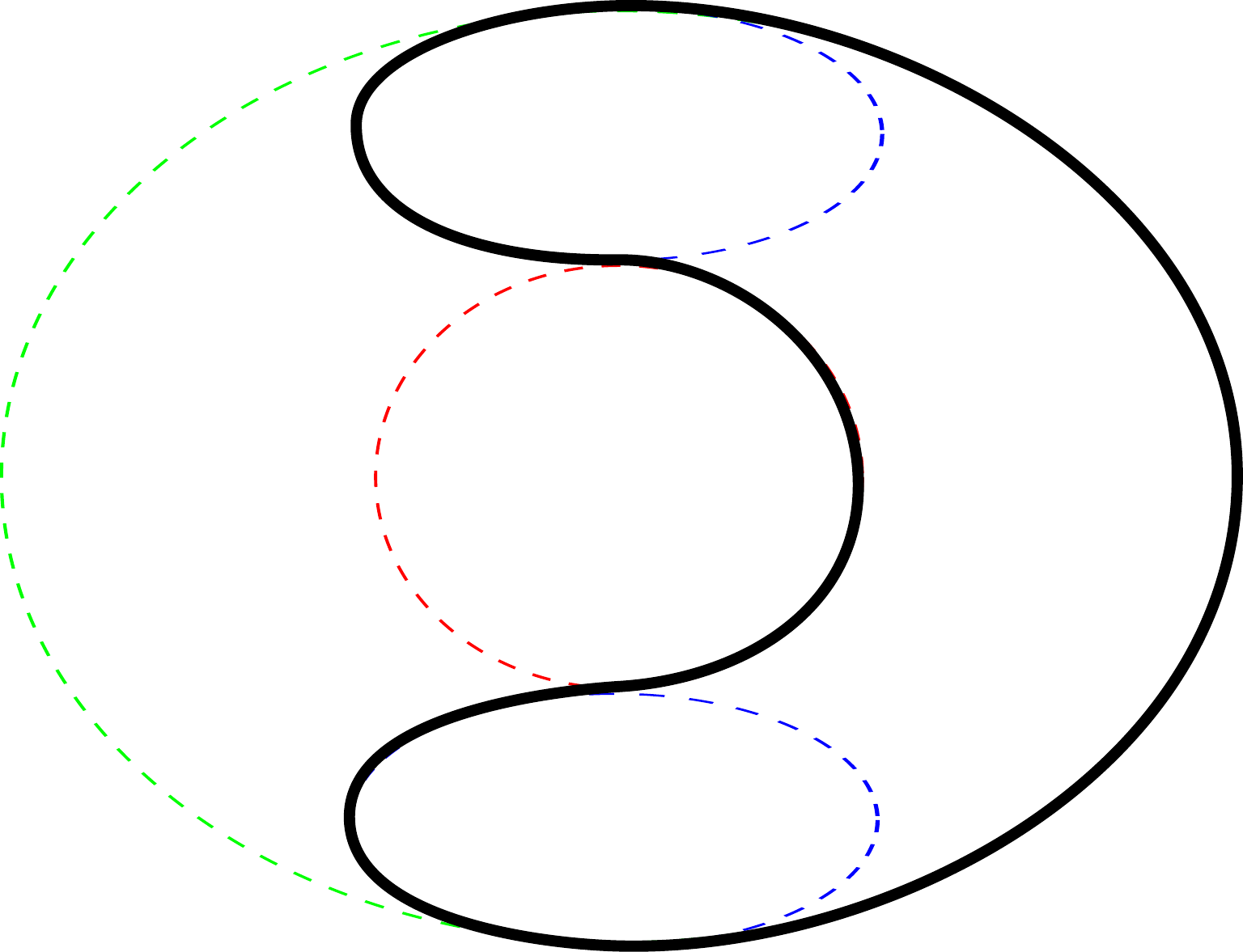}
\caption{Initial curve}\label{fig:initialcurves}
\end{figure}

\begin{align}
\label{paramstart}(4\cos(\theta),3\sin(\theta)) &\text{ for } -\frac{\pi}{2}\leq \theta\leq \frac{\pi}{2} \\
(2\cos(\theta),\frac{3}{4}\sin(\theta)+\frac{9}{4}) &\text{ for } \frac{\pi}{2}\leq \theta\leq \frac{3\pi}{2} \\
\label{zetaeqn} (\zeta \cos(\theta),\frac{3}{2} \sin(\theta)) &\text{ for } -\frac{\pi}{2} \leq \theta \leq \frac{\pi}{2}\\
\label{paramend} (2\cos(\theta), \frac{3}{4}\sin(\theta)-\frac{9}{4}) &\text{ for } \frac{\pi}{2}\leq \theta \leq \frac{3\pi}{2}.
\end{align}
The parameter $\zeta$ determines the depth of the non-convex well and is used as our measure of initial asymmetry of the curve.

\begin{figure}[h!]
\centering
\includegraphics[scale=.3]{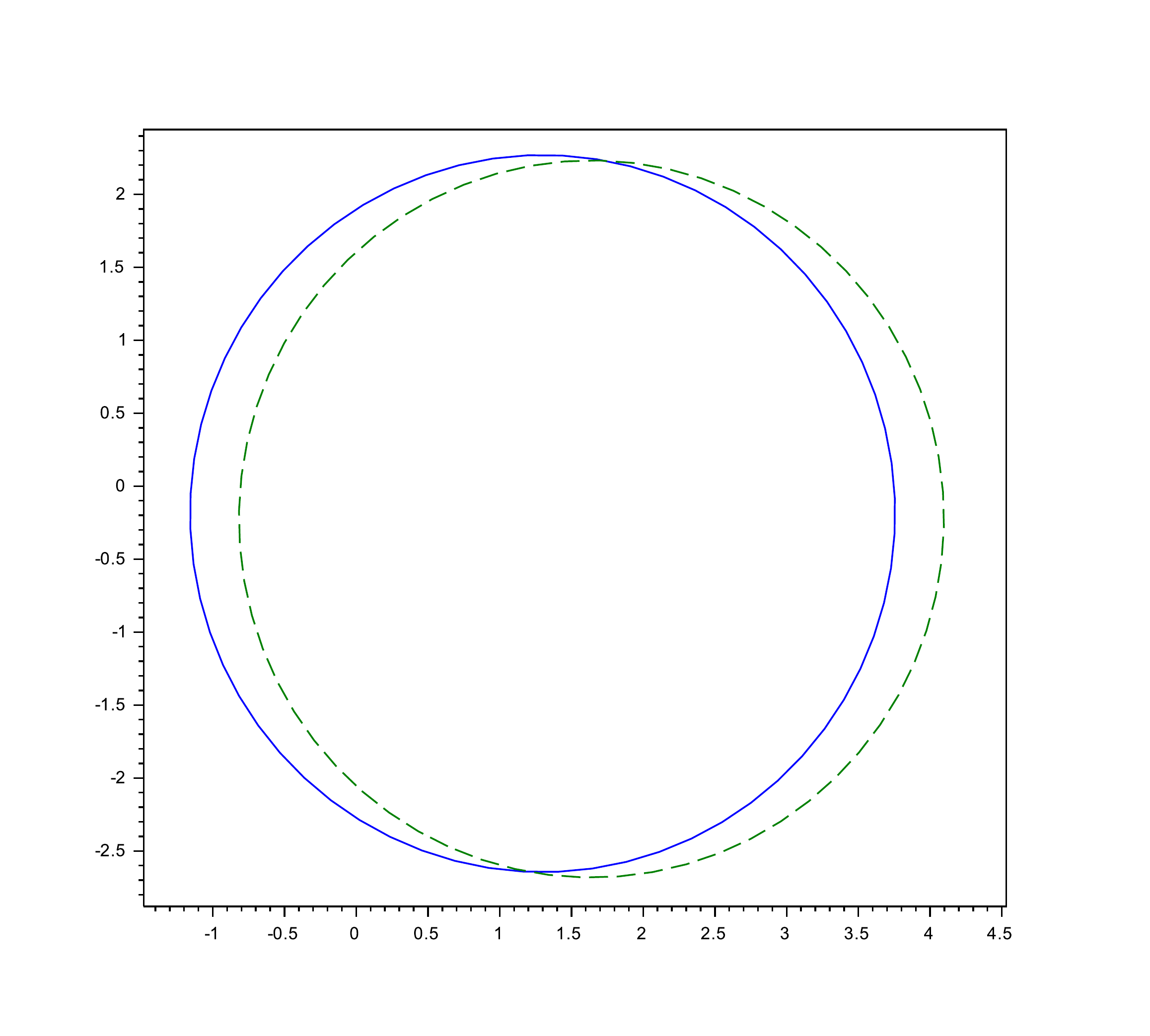}
\caption{Overall drift of curve with $\beta=1$ (blue, solid) compared to $\beta=0$ (green, dashed), starting from initial data \eqref{paramstart}-\eqref{paramend} with $\zeta=2$}\label{shift}
\end{figure}

To study the effect of $\beta$ and asymmetry on the overall motion of the curve, we measure the total transient motion by the following notion of the {\em drift of $\Gamma_\beta$}. First fix $\Gamma(0)$. Note that observation 1 implies that for sufficiently large time $k\Delta t$,  $\Gamma_\beta(k\Delta t)$ and $\Gamma_0(k\Delta t)$ will both be steady state circles. Define the {\em drift of $\Gamma_\beta$} to be the distance between the centers of these two circles.
\begin{remark} 
Note that this definition is used in order to account for numerical errors which may accumulate over time. Numerical drift of the center of mass of the curve caused by errors/approximations is offset by ``calibrating'' to the $\beta=0$ case.
\end{remark}
We consider the following two numerical tests:
\begin{enumerate}
\item {\em Dependence of drift on $\beta$:} Starting with the initial profile \eqref{paramstart}-\eqref{paramend} with $\zeta=1$, compute the drift of $\Gamma_{\beta}$ for various values of $\beta$.
\item {\em Dependence of drift on initial asymmetry:} Starting with the initial curve \eqref{paramstart}-\eqref{paramend} with $\beta=1$, compute the drift of $\Gamma_\beta$ for various values of $\zeta$.
\end{enumerate}
Taking $T=20$ is sufficient for simulations to reach circular steady state. We observe that drift increases with respect to $\zeta$ and increases linearly with respect to $\beta$. These data are recorded in Figure \ref{betatest}.

%

\begin{figure}[h!]
\centering
\includegraphics[width = .45\textwidth]{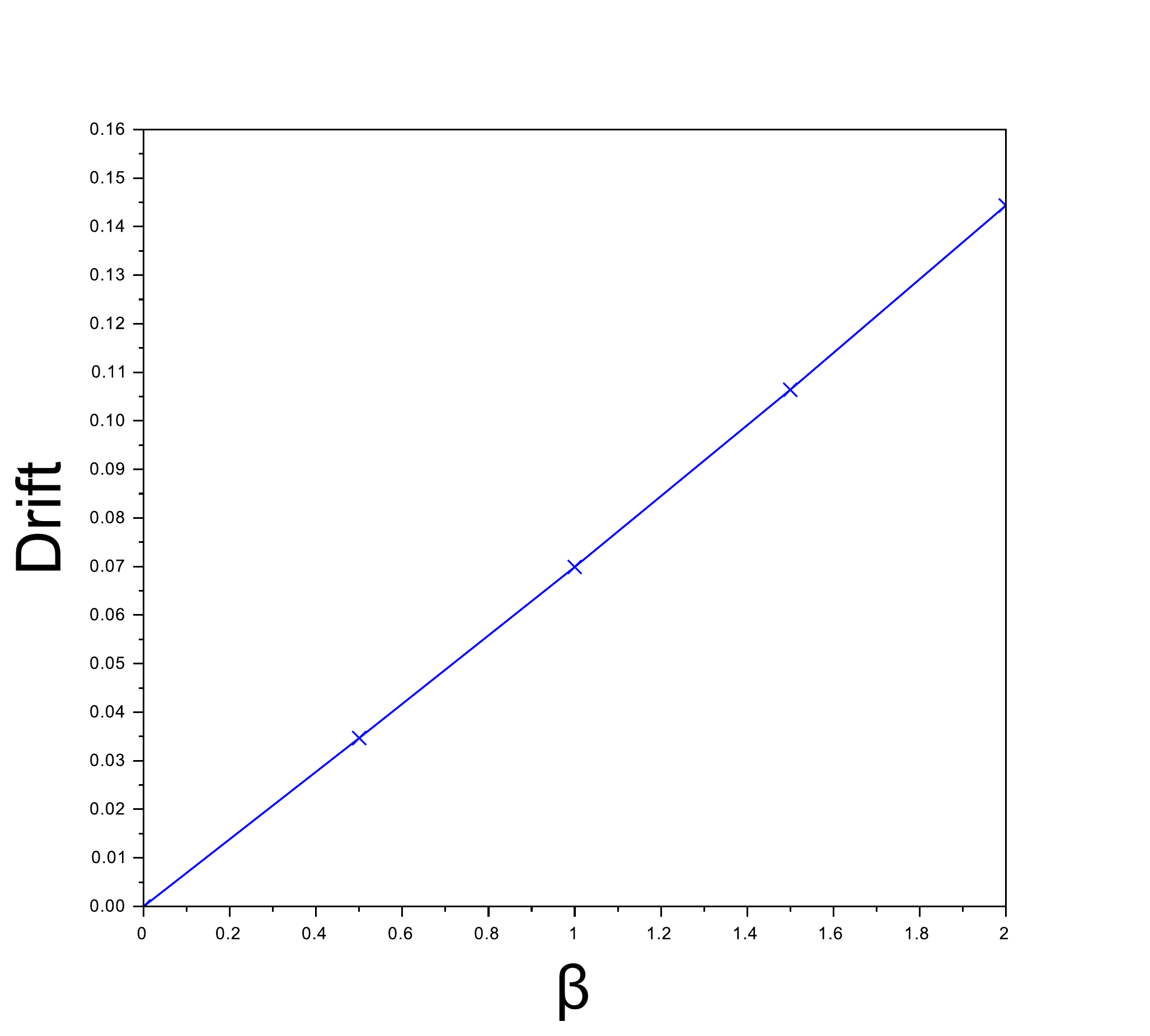}
\includegraphics[width = .45\textwidth]{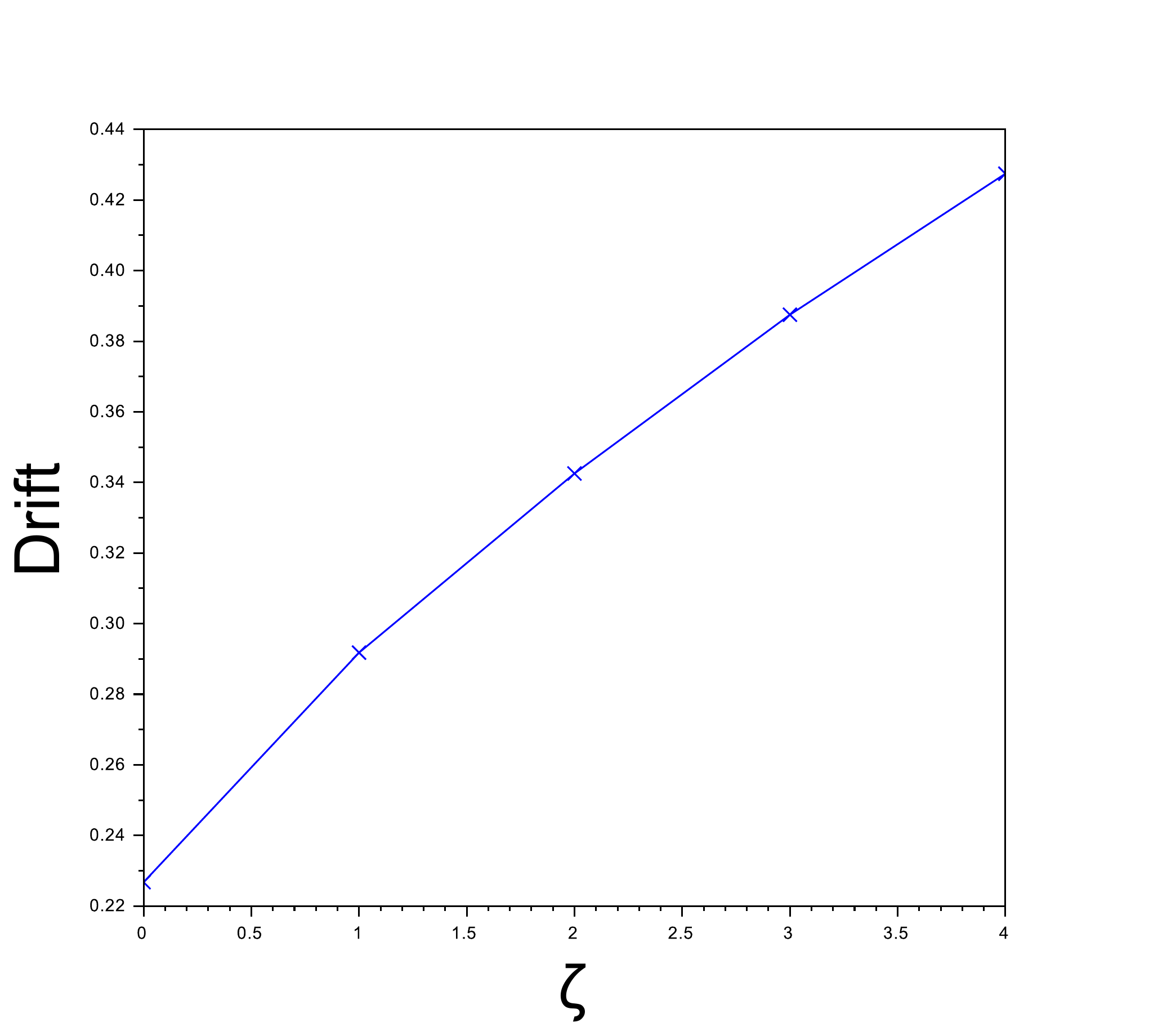}
\caption{Dependence of drift on the parameters $\beta$ and $\zeta$}\label{betatest}
\end{figure}

\medskip


\bibliographystyle{siam}
\bibliography{cellref}

\begin{thebibliography}{10}

\bibitem{Ber14}
{\sc L.~Berlyand, M.~Potomkin, and V.~Rybalko}, {\em Non-uniqueness in a
  nonlinear sharp interface model of cell motility}, arXiv:1409.5925v1,
  (2014).

\bibitem{Bon00}
{\sc A.~Bonami, D.~Hilhorst, and E.~Logak}, {\em Modified motion by mean
  curvature: local existence and uniqueness and qualitative properties},
  Differential Integral Equations, 13 (2000), pp.~1371--1392.

\bibitem{Bra78}
{\sc K.~A. Brakke}, {\em The motion of a surface by its mean curvature},
  Princeton University Press and University of Tokyo Press, 1978.

\bibitem{Doc76}
{\sc M.~P.~D. Carmo}, {\em Differential geometry of curves and surfaces},
  Pearson, 1976.

\bibitem{Che93}
{\sc X.~Chen}, {\em The {H}ele-{S}haw problem and area-preserving
  curve-shortening motions}, Arch. Rational Mech. Anal., 123 (1993),
  pp.~117--151.

\bibitem{Che10}
{\sc X.~Chen, D.~Hilhorst, and E.~Logak}, {\em Mass conserving {A}llen-{C}ahn
  equation and volume preserving mean curvature flow}, Interfaces Free Bound.,
  12 (2010), pp.~527--549.

\bibitem{Ell97}
{\sc C.~M. Elliott and H.~Garcke}, {\em Existence results for diffusive surface
  motion laws}, Adv. Math. Sci. Appl., 7 (1997), pp.~467--490.

\bibitem{Esc98}
{\sc J.~Escher and G.~Simonett}, {\em The volume preserving mean curvature flow
  near spheres}, Proc. Amer. Math. Soc., 126 (1998), pp.~2789--2796.

\bibitem{Gag86}
{\sc M.~Gage}, {\em On an area-preserving evolution equation for plane curves},
  Contemp. Math., 51 (1986), pp.~51--62.

\bibitem{GagHam86}
{\sc M.~Gage and R.~S. Hamilton}, {\em The heat equation shrinking convex plane
  curves}, J. Differential Geom., 23 (1986), pp.~69--96.

\bibitem{Gol94}
{\sc D.~Golovaty}, {\em The volume-preserving motion by mean curvature as an
  asymptotic limit of reaction-diffusion equations}, Quart. Appl. Math, 55
  (1997), pp.~243--298.

\bibitem{Gra87}
{\sc M.~A. Grayson}, {\em The heat equation shrinks embedded plane curves to
  round points}, J. Differential Geom., 26 (1987), pp.~285--314.

\bibitem{Ker08}
{\sc K.~Keren, Z.~Pincus, G.~M. Allen, E.~L. Barnhart, G.~Marriott,
  A.~Mogilner, and J.~A. Theriot}, {\em Mechanism of shape determination in
  motile cells}, Nature, 453 (2008), pp.~475--480.

\bibitem{Lad68}
{\sc O.~A. Ladyzenskaja, V.~A. Solonnikov, and N.~N. Ural'ceva}, {\em Linear
  and quasi-linear equations of parabolic type}, The American Mathematical
  Society, 1968.

\bibitem{Lio72}
{\sc J.~L. Lions and E.~Magenes}, {\em Non-homogeneous boundary value problems
  and applications}, Springer-Verlag, 1972.

\bibitem{Mer93}
{\sc B.~Merriman, J.~Bence, and S.~Osher}, {\em Diffusion generated motion by
  mean curvature motion}, in AMS Select Lectures in Mathematics: The
  Computational Crystal Grower's Workshop, J.~Taylor, ed., Am. Math. Soc.,
  1993.

\bibitem{Moh03}
{\sc R.~R. Mohan, A.~E.~K. Hutcheon, R.~Choi, J.~Hong, J.~Lee, R.~R. Mohan,
  R.~A. Jr., J.~D. Zieske, and S.~E. Wilson}, {\em Apoptosis, necrosis,
  proliferation, and myofibroblast generation in the stroma following {LASIK}
  and {PRK}}, Exp. Eye. Res., 76 (2003), pp.~71--87.

\bibitem{Osh88}
{\sc S.~Osher and J.~A. Sethian}, {\em Fronts propagating with curvature
  dependent speed: algorithms based on {H}amilton-{J}acobi formulations}, J.
  Comput. Phys., 79 (1988), pp.~12--49.

\bibitem{Ruu03}
{\sc S.~J. Ruuth and B.~T.~R. Wetton}, {\em A simple scheme for
  volume-preserving motion by mean curvature}, J. Sci. Comput., 19 (2003),
  pp.~373--284.

\bibitem{Sme03}
{\sc P.~Smereka}, {\em Semi-implicit level set methods for curvature and
  surface diffusion motion}, J. Sci. Comput., 19 (2003), pp.~439--456.

\bibitem{Zie12}
{\sc F.~Ziebert, S.~Swaminathan, and I.~Aranson}, {\em Model for
  self-polarization and motility of keratocyte fragments}, J. R. Soc.
  Interface, 9 (2012), pp.~1084--1092.

\end{thebibliography}

\end{document}